\documentclass[a4paper,leqno,12pt, reqno]{amsart}

\usepackage{amsmath}
\usepackage{amssymb}
\usepackage{graphicx}
\usepackage{tikz-cd}
\usepackage{hyperref}
\usepackage[all]{xy}

\usepackage{xspace}
\usepackage{bm}
\usepackage{amsmath}
\usepackage{amstext}
\usepackage{amsfonts}
\usepackage[mathscr]{euscript}
\usepackage{amscd}
\usepackage{latexsym}
\usepackage{amssymb}
\usepackage{enumerate}
\usepackage{xcolor}

\theoremstyle{plain}
\swapnumbers
\newtheorem{theorem}{Theorem}[section]
\newtheorem{proposition}[theorem]{Proposition}
\newtheorem{lemma}[theorem]{Lemma}
\newtheorem{corollary}[theorem]{Corollary}

\newtheorem*{thma}{Theorem A}
\newtheorem*{thmb}{Theorem B}
\newtheorem*{thmc}{Theorem C}
\newtheorem*{thmd}{Theorem D}

\newenvironment{mysubsection}[2][]
{\begin{subsec}\begin{upshape}\begin{bfseries}{#2.}
			\end{bfseries}{#1}}
		{\end{upshape}\end{subsec}}

\theoremstyle{definition}
\newtheorem{definition}[theorem]{Definition}

\newtheorem{examples}[theorem]{Examples}

\newtheorem{ack}[theorem]{Acknowledgements}
\newtheorem{subsec}[theorem]{}

\theoremstyle{remark}

\newcommand{\bX}{{\bf X}}

\newcommand{\Aut}{\operatorname{Aut}}
\newcommand{\AutG}{\Aut^G}

\newcommand{\aut}{\operatorname{aut}}
\newcommand{\map}{\operatorname{map}}
\newcommand{\Hom}{\operatorname{Hom}}
\newcommand{\Id}{\operatorname{Id}}
\newcommand{\Ker}{\operatorname{Ker}}
\newcommand{\nprod}{\bX}

\begin{document}
\title{Equivariant Self-homotopy equivalences of product spaces}
\author{Gopal Chandra Dutta}
\author{Debasis Sen}
\author{Ajay Singh Thakur}

\address{Department of Mathematics and Statistics\\
Indian Institute of Technology, Kanpur\\ Uttar Pradesh 208016\\India}

\email{gopald@iitk.ac.in}
	
\address{Department of Mathematics and Statistics\\
Indian Institute of Technology, Kanpur\\ Uttar Pradesh 208016\\India}

\email{debasis@iitk.ac.in}

\address{Department of Mathematics and Statistics\\
Indian Institute of Technology, Kanpur\\ Uttar Pradesh 208016\\India}

\email{asthakur@iitk.ac.in}

\date{\today}
	
	
\begin{abstract}
Let $G$ be a finite group. We study the group of $G$-equivariant self-homotopy equivalences of product of $G$-spaces. For a product of $n$-spaces, we represent it as product of $n$-subgroups under the assumption of equivariant reducibility. Further we describe each factor as a split short exact sequence. Also we obtain an another kind of factorisation, called $LU$ type decomposition, as product of two subgroups. 
\end{abstract}
\maketitle

\maketitle
	
\section{Introduction}
For a pointed space $X$, let $\Aut(X)$ denote the set of homotopy classes of based self homotopy equivalences. This set is a group under composition of maps \cite{Gshert, Oyse}. A natural approach for understanding $\Aut(X)$ would be to start from some decomposition of $X$, for example as a product or wedge of simpler spaces, and then describe self-homotopy equivalences of $X$ by in terms of those of the factors. The group $\Aut(X\times Y)$ was studied in the past by many authors (cf. \cite{ AYsep, Gps, PShp, PRshe, Sheps, ST}). In particular Sawashita considered the group $\Aut(X\times Y)$ when $X$ and $Y$ are spheres (\cite{Sheps}).  We have projections $p_X: X\times Y \to X, ~~p_Y: X \times Y\to Y$ and inclusions $\iota_X: X\hookrightarrow X\times Y$ and $\iota_Y: Y \hookrightarrow X\times Y$, determined by the base point. For a map $f: X\times Y \to X\times Y,$ let $$f_{XX}:=p_{X}\circ f\circ \iota_{X} : X\to X, ~~ f_{YY} := p_{Y}\circ f\circ \iota_{Y}: Y\to Y.$$ 
A self homotopy equivalence $f$ is called \emph{reducible} if $f_{XX}, f_{YY}$ are self-homotopy equivalences. When self-homotopy equivalences of $X\times Y$ are reducible, Heath and Booth obtained a split exact sequence 
$$ 0\to [X,\aut_{1}(Y)] \to \Aut(X\times Y)\xrightarrow{} \Aut(X)\times \Aut(Y)\to 0, $$
under the condition $[Y, \aut_1(X)] = \{\ast\}$ (\cite{Gps, Hgsep}). Here $\aut_1(X)$ is the connected component of $\Id$ of the space of based self homotopy equivalences of $X$. The group $\Aut(X\times Y)$ has two natural subgroups $\Aut_X(X\times Y)$ and $\Aut_Y(X\times Y)$, which are self-equivalences of $X\times Y$ over $X$ and $Y$ respectively. Later Pave\v{s}i\'{c} showed that under reducibility (without the condition $[Y, \aut_1(X)] = \{\ast\}$), the group $\Aut(X\times Y)$ can be decomposed as product of the subgroups $\Aut_X(X\times Y)$ and $\Aut_Y(X\times Y)$  (see \cite{PShp, PRshe}),

\begin{equation}\label{eqdecom}
\Aut(X\times Y) = \Aut_X(X\times Y)\cdot \Aut_Y(X\times Y).
\end{equation}
\noindent Further each of the factors fits into split exact sequences
 $$0\to [X,\aut_{1}(Y)]\to \Aut_{X}(X\times Y)\xrightarrow{} \Aut(Y)\to 0,$$ 
 and  
$$0\to [Y,\aut_{1}(X)]\to \Aut_{Y}(X\times Y)\xrightarrow{} \Aut(X)\to 0.$$
The result of Heath and Booth can be deduced as corollary of it. Moreover the decomposition of Equation \ref{eqdecom} was generalised for $n$-product of spaces in terms of $n$-subgroups. Further a $LU$ kind of decomposition was obtained for higher products (\cite{POg, PRshe}). 

In this paper, we consider equivariant self homotopy equivalences of product spaces and obtain equivariant versions of the results mentioned about. Let $G$ be a finite group. A based $G$-space is a $G$-space with a point which is $G$-fixed. 
The space of based $G$-equivariant maps is denoted by $\map^{G}(X,Y)$. Let $[X,Y]_{G}$ denote the set of  $G$-homotopy classes of based maps. For a based $G$-space $X$, we denote the space of based equivariant self-homotopy equivalences of $X$ by $\aut^G(X)$. We denote by $\AutG(X)$ the set of $G$-homotopy classes of elements of $\aut^G(X)$. The composition of maps induces a group operation on $\Aut^G(X)$. We call it the group of equivariant self-homotopy equivalences of $X$.  This group has been studied by several authors (cf. \cite{FS, SEG}). Our goal here is to study the group $\Aut^G(X\times Y)$, where $X, Y$ are based $G$-CW complexes and also for higher products. For our results we need the condition of $G$-reducibility.

%
\begin{definition}
With notations as above, a $G$-self-homotopy equivalence $f$ of $X\times Y$ is called \emph{$G$-reducible} if $f_{XX} \in \aut^G(X)$ and $f_{YY} \in \aut^G(Y)$.
\end{definition}

\noindent As in the non-equivariant case, the group $\AutG(X\times Y)$ has two natural subgroups $\AutG_X(X\times Y) = \{ f\in \AutG(X\times Y):~~ p_X\circ f = p_X \}$ and $\AutG_Y(X\times Y) = \{ f\in \AutG(X\times Y):~~ p_Y\circ f = p_Y \}$. We show that the product of these two subgroups is the whole group: 
\begin{thma}\label{thma}
If all $G$-self-homotopy equivalences of $X\times Y$ are $G$-reducible then
 $$\AutG(X\times Y)=\AutG_{X}(X\times Y)\cdot\AutG_{Y}(X\times Y).$$ 
 \end{thma}
 \noindent (See Theorem \ref{Mgred} below.)
 
 Under the assumption of $G$-reducibility, we have a natural map $\phi: \AutG(X\times Y)\to \AutG(X)\times \AutG(Y)$ defined by $\phi(f) = (f_{XX}, f_{YY})$. We show that the map $\phi$ gives a split exact sequence. This is an equivariaint analogue of Heath and Booth \cite{Gps}.
 
 \begin{thmb}\label{thmb}
Let $X, Y$ to be pointed connected $G$-CW complexes. If all $G$-self-homotopy equivalences of $X\times Y$ are $G$-reducible and $[Y,\aut_{1}(X)]_{G}=\{\ast \}$, then   we get a split short exact sequence
  $$0\to [X,\aut_{1}(Y)]_{G}\to \AutG(X\times Y)\xrightarrow[]{\phi} \AutG(X)\times \AutG(Y)\to 0.$$ 
  \end{thmb}
  \noindent (See Theorem \ref{t2prodexact} below.)

 We use the above results of two products to generalize for the $n$-product of spaces $\nprod = X_1\times \ldots \times X_n$, where all $X_{i} $'s are $G$-spaces. The action of $G$ on $\nprod$ is the diagonal action. Let    $\prod_{\hat{i}}:=X_{1}\times \ldots \times \hat{X_{i}}\times \ldots \times X_{n}$, i.e. $\prod_{\hat{i}}$ denotes the subproduct of $\nprod$ obtained by omitting $X_{i}$. We have a natural projection $\nprod \to \prod_{\hat{i}}$. The set of $G$-homotopy classes of $G$-self-equivalences of $\nprod$ over $\prod_{\hat{i}}$ is denoted by $\AutG_{\prod_{\hat{i}}}(\nprod)$. These are subgroups of $\Aut^G(\bX)$.

\begin{thmc}
If all $G$-self-homotopy equivalences of $\nprod$ are $G$-reducible, then  
$$\AutG(\nprod)=\AutG_{\prod_{\hat{n}}}(\nprod)\ldots \AutG_{\prod_{\hat{1}}}(\nprod).$$ 
Moreover each $\AutG_{\prod_{\hat{i}}}(\nprod)$ fits into a split exact sequence,
$$0\to \Ker(\phi_{i})\to \AutG_{\prod_{\hat{i}}}(\nprod)\xrightarrow[]{\phi_{i}}\AutG(X_{i})\to 0,~~~ 1\leq i \leq n.$$
\end{thmc}
\noindent (See Theorem \ref{tnprod} below.)

 The group $\AutG(\nprod)$ can be expressed as the product of its $n$-subgroups $\AutG_{\prod_{\hat{i}}}(\nprod)$ (where $i=1,\ldots,n$), which is difficult from computational viewpoint. Therefore we develop a new factorization called LU decomposition, which is the following: 

\begin{thmd}
If all $G$-self-homotopy equivalences of $\nprod$ are $G$-reducible, then $$\AutG(\nprod)=L^{G}(X_{1},\ldots,X_{n})\cdot U^{G}(X_{1},\ldots,X_{n}).$$
\end{thmd}
(See Theorem \ref{tlu} below.) The factor $L^{G}(X_{1},\ldots,X_{n})$ has the property that the induced maps on the homotopy groups can be represented by lower triangular matrices and similarly by upper triangular matrices for the other factor.
%
%
\begin{mysubsection}{Organisation}
The rest of the paper is organised as follows. In Section \ref{sproducttwo} we consider product of two spaces. We give some sufficient conditions of $G$-reducibility  and obtain Theorem A and Theorem B. In Section \ref{snprod}, we consider the case of $n$-products of spaces and obtain Theorem C. In the last Section \ref{slu}, we obtain the $LU$-decomposition. 

Throughout the paper, $G$ will denote a finite group and all spaces are based $G$-CW complexes. All maps are base point preserving and equivariant.

\begin{ack}
The first author would like to thank IIT Kanpur for PhD fellowship.
\end{ack}

\end{mysubsection}

\begin{section}{$\AutG(X\times Y)$}\label{sproducttwo}

We begin this section by recalling some basic facts about $G$-CW complexes. We refer to \cite{MEQ} for details. Let $X$ be a $G$-space. For a subgroup $H< G$, the $H$-fixed point set is defined as $X^H = \{x\in X: ~~ hx = x,~~ \forall h\in H\}$. If $f: X\to Y$ is a $G$-map then $f^H := f |_{X^H} : X^H \to Y^H$.  We know that an equivariant map $f: X \to Y$ between $G$-CW complexes is $G$-homotopy equivalence if and only if $f^{H}: X^H \to Y^H$ are ordinary homotopy equivalences for all subgroups $H$ of $G$.  	

\begin{mysubsection}{$G$-reducibility}
For $G$-CW complexes $X,Y$, an equivariant self-homotopy equivalence $f$ of $X\times Y$ is $G$-reducible if and only if $f^{H}$ is a reducible self-homotopy equivalence of $X^H\times Y^H$ for all subgroups $H$ of $G$. 
The following are some examples of when $G$-self homotopy equivalences of $X\times Y$ are $G$-reducible.
\begin{examples}	
Let $X,Y$ be two pointed $G$-CW complexes.
\begin{enumerate}[(a)]
	
\item Assume that for every positive integer $n$ and for every pair of $G$-maps $f:X\to Y, g:Y\to X$, at least one of the induced maps $\pi_{n}(f^{H}):\pi_{n}(X^{H})\to \pi_{n}(Y^{H}), \pi_{n}(g^{H}):\pi_{n}(Y^{H})\to \pi_{n}(X^{H})$ is trivial for all $H\leq G$. Then all $G$-self equivalences of $X\times Y$ are $G$-reducible.

To see this, let $h\in \AutG(X\times Y)$. Therefore $h^{H}:X^{H}\times Y^{H}\to X^{H}\times Y^{H}$ is an ordinary homotopy equivalence for all $H\leq G$. 
By the given condition at least one of the induced homomorphisms on the homotopy groups of the maps $(h^{H})_{XY}:Y^{H}\to X^{H},~~(h^{H})_{YX}:X^{H}\to Y^{H}$ is trivial.
Now from \cite[Proposition 2.1]{PShp}, we say that $(h^{H})_{XX}, (h^{H})_{YY}$  are self-homotopy equivalences of $X^{H}$ and $Y^{H}$ respectively for all $H\leq G$. Therefore $(h_{XX})^{H}$ and $(h_{YY})^{H}$ both are self-homotopy equivalences of $X^{H}$ and $Y^{H}$ respectively  for all $H\leq G$ (since $(h^{H})_{XX}=(p_{X})^{H}\circ h^{H}\circ (i_{X})^{H}=(p_{X}\circ h\circ i_{X})^{H}=(h_{XX})^{H}$. Consequently $h_{XX}\in \AutG(X), h_{YY}\in \AutG(Y)$.

\item For every positive integer $n$, if at least one of the group $\pi_{n}(X^{H}), \pi_{n}(Y^{H})$ is trivial for all $H\leq G$, then all $G$-self equivalences of $X\times Y$ are $G$-reducible. 

This follows from the above one.

\item If at least one of ~~$\Hom(\pi_{n}(X^{H}),\pi_{n}(Y^{H})), \Hom(\pi_{n}(Y^{H}), \pi_{n}(X^{H}))$ is trivial $ \forall ~   H\leq G, \forall n>0$ then all $G$-self equivalences of $X\times Y$ are $G$-reducible.

This also immediately follows from the above.

\item Let $m,n\geq 2$. Assume that the group $G = \mathbb{Z}_{2}$ acts on $S^{m}$ and $S^{n}$ by reflection about $m$ and $n$ dimensional hyperplane respectively. If $m\neq n$, then all $\mathbb{Z}_{2}$-self equivalences of $S^{m}\times S^{n}$ are $\mathbb{Z}_{2}$-reducible. To see this, observe that,

\[(S^{m})^{H} = \begin{cases}
		S^{m-1} ,~~~ if ~H=\mathbb{Z}_{2} \\
	S^{m},~~~~~~~  if~~H=\{1\}.
\end{cases} \] 
\noindent If $m<n$, then $[(S^{m})^{H},(S^{n})^{H}]$ is trivial. Therefore the induced maps between $\pi_{k}(S^{m})^{H}$ and $\pi_{k}(S^{n})^{H}$ are trivial, for all $H\leq \mathbb{Z}_{2}$ and for all $k>0$. Similarly for $m>n$, we can say that the induced maps between $\pi_{k}(S^{n})^{H}$ and $\pi_{k}(S^{m})^{H}$ are trivial, for all $H\leq \mathbb{Z}_{2}$ and for all $k>0$.
Hence all $\mathbb{Z}_{2}$-self-homotopy equivalences of $S^{m}\times S^{n}$ are $\mathbb{Z}_{2}$-reducible.
\end{enumerate}			
\end{examples}

\noindent Recall that, an endomorphism $\phi\colon \Gamma\to \Gamma$ of an additively written group $\Gamma$ is called \emph{quasi-regular} if the map $(\Id-\phi)\colon \Gamma\to \Gamma$ defined by $\gamma\mapsto \gamma- \phi(\gamma)$, is an automorphism. A self $G$-map $f\colon X\to X$ is called \emph{quasi-regular} if the endomorphism of groups $(f^{H})_{\#n}\colon \pi_{n}(X^{H})\to \pi_{n}(X^{H})$ are quasi-regular for all $H\leq G$ and all $n$. An endomorphism $\phi\colon \Gamma\to \Gamma$ is called \emph{nilpotent} if $\phi^{k}=0$ for some $k$. For a nilpotent endomorphism $\phi\colon \Gamma\to \Gamma$ such that $(\Id-\phi)$ is a homomorphism, then $\phi$ is quasi-regular (cf. \cite{PRshe}).

The following Proposition gives some sufficient conditions for $G$-reducibility. We denote $f_{XY} = p_X\circ f\circ \iota_Y: Y \to X$ and $f_{YX} = p_Y\circ f\circ \iota_X: X \to Y.$
\begin{proposition}\label{pgred}		
Let  $X, Y$ be based $G$-CW complexes.
\begin{enumerate}[(a)]
\item Assume that  $f_{YY}\in \AutG(Y)$ for all $f\in \AutG(X\times Y)$. Then all $G$- self-homotopy equivalences of $X\times Y$ are $G$-reducible.
\item Assume that for any $f\in \AutG(X\times Y)$ one of the maps $f_{XY}\circ \widetilde{f}_{YX}: X \to X$ and $f_{YX}\circ \widetilde{f}_{XY}: Y\to Y$ is quasi-regular. Then all $G$-self-homotopy equivalences of $X\times Y$ are $G$-reducible. Here $\widetilde{f}$ is the $G$-hmotopy inverse of $f$ in $\AutG(X\times Y)$.
\item Suppose that every self $G$-map of $X$, which can be factored through $Y$ (up to homotopy), induces a nilpotent endomorphism on $\pi_n(X^H)$ for all $n,~ H<G$ (i.e. if $X$ and $Y$ are homotopically distant), then $G$-self-homotopy equivalences of $X\times Y$ are $G$-reducible.
\end{enumerate}
\end{proposition}
\begin{proof}
\begin{enumerate}[(a)]
	
\item Let $f\in \AutG(X\times Y)$. Then by assumption $f_{YY}\in \AutG(Y)$. We have to show that $f_{XX}\in \AutG(X)$ which is equivalent to show that $(f_{XX})^{H}\in \Aut(X^{H}),~~ \forall H\leq G$. By definition $(f_{X})^{H}\circ (\iota_{X})^{H}=(f_{XX})^{H}$. Therefore by \cite[Proposition 2.3(b)]{PShp}, it is enough to show $((f_{X})^{H},(p_{Y})^{H})=(f_{X},p_{Y})^{H}\in \Aut(X^{H}\times Y^{H})$ for all $H\leq G$.

Since $\AutG(X\times Y)$ is a group, $f = ({f}_{X},{f}_{Y})$ has an inverse $\bar{f} = (\bar{f}_{X},\bar{f}_{Y}) \in \AutG(X\times Y)$. So $\bar{f}\circ f=\Id_{X\times Y}$.
Now observe that, 
		
$$(p_X, p_Y) = \Id_{X\times Y} = \bar{f}\circ f= (\bar{f}_{X},\bar{f}_{Y})\circ (f_{X},f_{Y})= (\bar{f}_{X}\circ (f_{X},f_{Y}),\bar{f}_{Y}\circ (f_{X},f_{Y})).$$
So we get, $\bar{f}_{X}\circ (f_{X},f_{Y})=p_{X}; ~~ \bar{f}_{Y}\circ (f_{X},f_{Y})=p_{Y}.$  By the given condition, we can say that $\bar{f}_{YY}\in \AutG(Y)$, as $\bar{f}\in \AutG(X\times Y)$. Hence $(\bar{f}_{YY})^{H}\in \Aut(Y^{H}), ~\forall H\leq G$, i.e. $(\bar{f}_{Y})^{H}\circ (\iota_{Y})^{H}\in \Aut(Y^{H})$ for all $H\leq G$. Therefore from \cite[Proposition 2.3(b)]{PShp}, we have $(p_{X},\bar{f}_{Y})^{H}=((p_{X})^{H},(\bar{f}_{Y})^{H})\in \Aut(X^{H}\times Y^{H}),~~ \forall H\leq G$. Hence $(p_{X},\bar{f}_{Y})\in \AutG(X\times Y)$.
Now, $$(f_{X},p_{Y})=(f_{X},\bar{f}_{Y}\circ (f_{X},f_{Y}))=(p_{X},\bar{f}_{Y})\circ (f_{X},f_{Y})\in \AutG(X\times Y).$$ Hence $(f_{X},p_{Y})^{H}\in \Aut(X^{H}\times Y^{H})$ for all $H\leq G$.

\item For any map $f\colon X\times Y \to X\times Y$ the induced endomorphism on homotopy groups $(f^H)_{\#n}\colon {\pi}_{n}(X^H)\times {\pi}_{n}(Y^H)\to {\pi}_{n}(X^H)\times {\pi}_{n}(Y^H)$ can be represented by the matrix form for all $H\leq G$,
$$ M_{n}(f^{H}):=
\begin{bmatrix}
	(f^{H}_{XX})_{\#n}& (f^{H}_{XY})_{\#n}\\
	(f^{H}_{YX})_{\#n}& (f^{H}_{YY})_{\#n}
\end{bmatrix}.$$
Without any ambiguity we can write $(f^{H})_{XX} = (f_{XX})^{H} = f^{H}_{XX}$ and so on.

\noindent Therefore for any $f\in \AutG(X\times Y)$ with its inverse $\widetilde{f}$, we can say that $$M_{n}(f^{H})\cdot M_{n}(\widetilde{f^{H}}) = M_{n}(\Id_{X^{H}\times Y^{H}}).$$
This implies that $(f^{H}_{YX})_{\#n}\circ (\widetilde{f}^{H}_{XY})_{\#n} + (f^{H}_{YY})_{\#n}\circ (\widetilde{f}^{H}_{YY})_{\#n} = \Id_{\pi_{n}(Y^{H})}$. So we get, $$(f^{H}_{YY})_{\#n}\circ (\widetilde{f}^{H}_{YY})_{\#n} = \Id_{\pi_{n}(Y^{H})} - (f^{H}_{YX})_{\#n}\circ (\widetilde{f}^{H}_{XY})_{\#n} =  \Id_{\pi_{n}(Y^{H})} - (f_{YX}\circ \widetilde{f}_{XY})^{H}_{\#n}.$$
If $f_{YX}\circ \widetilde{f}_{XY}$ is quasi-regular, then the maps $\Id_{\pi_{n}(Y^{H})} - (f_{YX}\circ f_{XY})^{H}_{\#n}$ are automorphsims for all $n$ and $H\leq G$. Consequently, $(f^{H}_{YY})_{\#n}\colon \pi_{n}(Y^{H})\to \pi_{n}(Y^{H})$ are isomorphsim for all $n$ and all $H\leq G$. By Whitehead theorem we get $f_{YY}\in \AutG(Y)$. Using the part (a), we get that the $G$-self-homotopy equivalences of $X\times Y$ are $G$-reducible.

Similarly, if we take $f_{XY}\circ \widetilde{f}_{YX}$ is quasi-regular, then using the equality $$(f^{H}_{XY})_{\#n}\circ (\widetilde{f}^{H}_{YX})_{\#n} + (f^{H}_{XX})_{\#n}\circ (\widetilde{f}^{H}_{XX})_{\#n} = \Id_{\pi_{n}(X^{H})},$$ we get the desire result.

\item Let $f\colon X\times Y \to X\times Y$ be a self $G$-homotopy equivalence with $G$-homotopy inverse $\widetilde{f}\colon X\times Y\to X\times Y$. From part(b) its enough to show that $f_{XY}\circ \widetilde{f}_{YX}$ is quasi-regular, which is equivalent to show that $\Id_{\pi_{n}(X^{H})} - (f_{XY}\circ \widetilde{f}_{YX})^{H}_{\#n}$ are homomorphsims, whenever $(f_{XY}\circ \widetilde{f}_{YX})^{H}_{\#n}$ are nilpotent endomorphism for all $n$ and all $H\leq G$. Observe that, $\Id_{\pi_{n}(X^{H})} - (f_{XY}\circ \widetilde{f}_{YX})^{H}_{\#n} = (f_{XX})^{H}_{\#n}\circ (\widetilde{f}_{XX})^{H}_{\#n}$ are homorphisms for all $n$ and all $H\leq G$. By the given assumption, $(f_{XY}\circ \widetilde{f}_{YX})_{\#n}$ are nilpotent natural transformation for all $n$. Therefore $(f_{XY}\circ \widetilde{f}_{YX})^{H}_{\#n}$ are nilpotent endomorphism for all $n$ and all $H\leq G$, so the natural transformation $(f_{XY}\circ \widetilde{f}_{YX})_{\#n}$ are quasi-regular endomorphisms for all $n$. Hence the map $f_{XY}\circ \widetilde{f}_{YX}$ is quasi-regular and we get the $G$-reducibility of the map $f$.

\end{enumerate}
\end{proof}

\end{mysubsection}

Recall that we have natural projections $p_X: X\times Y\to X$ and $p_Y: X\times Y \to Y$. A self-map of $X\times Y$ such that $p_X\circ f = p_X$ looks like $(p_X, f_Y)$. We define $$\Aut^G_X (X\times Y) = \{ f\in \Aut^G(X\times Y)~|~~ p_X\circ f = p_X \},$$ $$\Aut^G_Y (X\times Y) = \{ f\in \Aut^G(X\times Y)~|~~ p_Y\circ f = p_Y \}.$$
\noindent We now show that these are subgroups $\AutG(X\times Y)$ and obtain a short exact sequence. Recall that $\aut_1(X)$ is the connected component of identity of $\aut(X)$. This has a $G$-action defined by $(g.f) (x) = gf(g^{-1}x).$

\begin{proposition}\label{egred}
\begin{enumerate}[(a)]
\item If $(p_{X},g):X\times Y \to X\times Y$ is a $G$-self-homotopy equivalence, then there exists a $G$-map $\widetilde{g}:X\times Y\to Y$, such that $(p_{X},\widetilde{g})$  is a $G$-equivariant homotopy inverse of $(p_{X},g)$.

\item The map $(p_{X},g):X\times Y \to X\times Y$ is a $G$-homotopy equivalence if and only if $g\circ \iota_{Y}\in \AutG(Y)$.

\item $\AutG_{X}(X\times Y)$ is a subgroup of $\AutG(X\times Y)$ whose elements are of the form $(p_{X},g)$ with $g\circ \iota_{Y}\in \AutG(Y)$.

\item If $X$ is a connected $G$-CW complex, then there exists a split exact sequence $$0\to [X,\aut_{1}(Y)]_{G}\to \AutG_{X}(X\times Y)\xrightarrow{\phi} \AutG(Y)\to 0, $$ 
\noindent where $\phi (p_{X},g)=g\circ \iota_{Y}$, for all $(p_{X},g)\in \AutG_{X}(X\times Y).$
\end{enumerate}
\end{proposition}
	
\begin{proof}
\begin{enumerate}[(a)]
\item Let $(p_{X},g)\in \AutG(X\times Y)$. Then it has an inverse in $\AutG(X\times Y)$, say $(f_{X},f_{Y})$.
So,
\begin{align*}
&(p_{X},g)\circ (f_{X},f_{Y})=\Id_{X\times Y}=(p_{X},p_{Y})\\
\Rightarrow &~(f_{X},g\circ(f_{X},f_{Y}))=(p_{X},p_{Y})\\
\Rightarrow &~f_{X}=p_{X}. 
\end{align*}
Therefore $(p_{X},f_{Y})$  is a $G$-equivariant homotopy inverse of $(p_{X},g)$. Consequently  $(p_{X},\widetilde{g})$  is a $G$-equivariant homotopy inverse of $(p_{X},g)$, where $\widetilde{g}=f_{Y}:X\times Y\to Y$.

\item  We know that $(p_{X},g):X\times Y \to X\times Y$ is a $G$-equivariant self-homotopy equivalence, if and only if $(p_{X},g)^{H}:X^{H}\times Y^{H}\to X^{H}\times Y^{H}$ are self-homotopy equivalences for all $H\leq G$, i.e $((p_{X})^{H},g^{H})\in \Aut(X^{H}\times Y^{H})$ for all $H\leq G$.
From \cite[Proposition 2.3]{PShp}, we say that  $((p_{X})^{H},g^{H})\in \Aut(X^{H}\times Y^{H})$  if and only if $g^{H}\circ (\iota_{Y})^{H}\in \Aut(Y^{H})$. Therefore  $(p_{X},g):X\times Y \to X\times Y$ is a $G$-equivariant homotopy equivalence, if and only if $(g\circ \iota_{Y})^{H}=g^{H}\circ (\iota_{Y})^{H}\in \Aut(Y^{H})$ for all $H\leq G$  i.e. $g\circ \iota_{Y}\in \AutG(Y)$.

\item  Follows from part (a) and (b).

\item  First we prove that $\phi$ is a homomorphism. Let $(p_{X},g),(p_{X},f)\in \AutG_{X}(X\times Y)$, then
 $$\phi ((p_{X},g)\circ (p_{X},f)) = \phi(p_{X},g\circ (p_{X},f)) = g\circ (p_{X},f)\circ \iota_{Y}$$ $$= g\circ \iota_{Y}\circ f\circ \iota_{Y} = \phi(p_{X},g)\circ \phi(p_{X},f).$$
Let us define an another map $$\psi:\AutG(Y)\to \AutG_{X}(X\times Y),  ~~\psi(h) = (p_{X},h\circ p_{Y}),$$ for all $h\in \AutG(Y)$. Now $(\phi \circ \psi) (h)=\phi(p_{X},h\circ p_{Y})=h\circ p_{Y}\circ \iota_{Y}=h$  for all $h\in \AutG(Y)$, i.e. $\phi\circ \psi = \Id$. Therefore $\phi$ is an onto map.

Now we determine the kernel of $\phi.$ Clearly, $(p_{X},f)\in \ker(\phi)$  if and only if $\phi(p_{X},f)=\Id_{Y}$, i.e $f\circ \iota_{Y}=\Id_{Y}$, where $f:X\times Y\to Y$ is a $G$-map. Therefore $(p_{X},f)\in \Ker(\phi)$ corresponds under the adjoint associativity to a base point preserving map $\widetilde{f}:X\to \map(Y,Y)$ defined as $\widetilde{f}(x)(y)=f(x,y)$. Observe that $\widetilde{f}$ sends the base point of $X$ to $\Id_{Y}$. Since $X$ is a connected, the image of $\widetilde{f}$ is contained in the component of $\map(Y,Y)$ that contains $\Id_{Y}$, denoted by $\aut_{1}(Y)$. Again $\widetilde{f}$ is a $G$-equivariant map, because $$(g\cdot \widetilde{f}(x))(y)=g\widetilde{f}(x)(g^{-1}y)=gf(x,g^{-1}y)=f(gx,y)=\widetilde{f}(gx)(y),$$ i.e. $\widetilde{f}(gx)=g\cdot \widetilde{f}(x)$. Therefore $\Ker(\phi)$ can be identified with $[X,\aut_{1}(Y)]_{G}$, which proves the assertion.
			
\end{enumerate}
\end{proof}

The following corollary follows easily from Proposition \ref{egred}(b).
\begin{corollary}\label{lgred}
Assume that all $G$-self-homotopy equivalences of $X\times Y$ are $G$-reducible. If $f=(f_{X},f_{Y})\in \AutG(X\times Y)$, then  $(p_{X},f_{Y})\AutG_{X}(X\times Y)$, and $(f_{X},p_{Y})\in \AutG_{Y}(X\times Y)$.
\end{corollary}

 The following is a sufficient condition for an equivariant self-map $f$ of $X\times Y$ to be a $G$-homotopy equivalence. 

\begin{proposition}\label{mgred}
	Let us assume all $G$-self-homotopy equivalences of $X\times Y$ are G-reducible. For any $f\in \map^{G}(X\times Y, X\times Y)$,  we have $f\in \AutG(X\times Y)$ if and only if $f_{XX}\in \AutG(X)$ and  $f_{YY}\in \AutG(Y)$.
\end{proposition}

\begin{proof}
The forward implication is trivial. 
	
Conversely, let $f\in \map^{G}(X\times Y, X\times Y)$ such that $f_{XX}\in \AutG(X), f_{YY}\in \AutG(Y)$. By Proposition \ref{egred}[(a), (b)], we have $(f_{X},p_{Y})\in \AutG(X\times Y)$  with inverse $(\widetilde{f_{X}},p_{Y})\in \AutG(X\times Y)$.
Observe that, 
$$f=(f_{X},f_{Y})=[p_{X},f_{Y}\circ (f_{X},p_{Y})^{-1}]\circ(f_{X},p_{Y})=[p_{X},f_{Y}\circ (\widetilde{f_{X}},p_{Y})]\circ(f_{X},p_{Y}).$$ 
Moreover, $$(p_{X},f_{Y})\circ (\widetilde{f_{X}},p_{Y})=[p_{X}\circ (\widetilde{f_{X}},p_{Y}), f_{Y}\circ (\widetilde{f_{X}},p_{Y})]\in \AutG(X\times Y).$$
By applying the $G$-reducibility to the right-hand side, we say that $f_{Y}\circ (\widetilde{f_{X}},p_{Y})\circ \iota_{Y}\in \AutG(Y)$. Therefore from Proposition \ref{egred}(b), we have $(p_{X},f_{Y}\circ (\widetilde{f_{X}},p_{Y}))\in \AutG(X\times Y)$.
Consequently, $f=[p_{X},f_{Y}\circ (\widetilde{f_{X}},p_{Y})]\circ(f_{X},p_{Y})\in \AutG(X\times Y)$.

\end{proof}

We have seen in Proposition \ref{egred} that $\AutG_{X}(X\times Y)$ and $\AutG_{Y}(X\times Y)$ are two subgroups of $\AutG(X\times Y).$ We now show that the product of these subgroups is the whole group.
\begin{theorem}\label{Mgred}
Let  us asssume all self-homotopy equivalences of $X\times Y$ are $G$-reducible. Then  $$ \AutG(X\times Y)=\AutG_{X}(X\times Y)\cdot \AutG_{Y}(X\times Y).$$
\end{theorem}
\begin{proof}
Clearly, $\AutG_{X}(X\times Y)\cap \AutG_{Y}(X\times Y)=(p_{X},p_{Y})$. Since the intersection of the two subgroups is trivial, therefore if the factorization of the element in $\AutG(X\times Y)$ exists, it will be unique. We have the factorization, $$f=(f_{X},f_{Y})=[p_{X},f_{Y}\circ (f_{X},p_{Y})^{-1}]\circ(f_{X},p_{Y})=[p_{X},f_{Y}\circ (\widetilde{f_{X}},p_{Y})]\circ(f_{X},p_{Y}).$$ From Corollary \ref{lgred} we can say that the second factor $(f_{X},p_{Y})$ in the decomposition of $f$ is in $\AutG_{Y}(X\times Y)$. Also its inverse $(\widetilde{f}_{X},p_{Y})\in \AutG(X\times Y)$. It is sufficient to show that the first factor is in $\AutG_{X}(X\times Y)$. Now $$(f_{X},f_{Y})\circ (\widetilde{f}_{X},p_{Y})=[f_{X}\circ (\widetilde{f}_{X},p_{Y}),f_{Y}\circ (\widetilde{f}_{X},p_{Y})]\in \AutG(X\times Y),$$ using $G$-reducibility we have, $f_{Y}\circ (\widetilde{f}_{X},p_{Y})\circ \iota_{Y}\in \AutG(Y)$. Therefore by Proposition \ref{egred}(b), we can say that the first factor $[p_{X},f_{Y}\circ (\widetilde{f_{X}},p_{Y})]\in \AutG_{X}(X\times Y)$.

\end{proof}

\begin{lemma}\label{Lgred}
Let Y be a connected $G$-CW complex such that $[Y,\aut_{1}(X)]_{G}=\{ \ast \}$. For every $G$-map $h\colon X\times Y\to X$  such that $h\circ \iota_{X}\in \AutG(X)$, then the relation $h\simeq h\circ \iota_{X}\circ p_{X}$  holds.
\end{lemma}

\begin{proof}
From \ref{egred}(d), we have a short exact sequence, $$0\to [Y,\aut_{1}(X)]_{G}\to \AutG_{Y}(X\times Y)\xrightarrow[]{\phi} \AutG(X)\to 0,$$ defined as $$\phi(f,p_{Y}):=f\circ \iota_{X}.$$
Since $[Y,\aut_{1}(X)]_{G}=\{\ast\}$, therefore $\phi$ is an isomorphism.
Let the splitting map of $\phi$ is $\psi\colon \AutG(X)\to \AutG_{Y}(X\times Y)$, defined as $\psi(g):=(g\circ p_{X},p_{Y})$.
Now for any $G$-map $h\colon X\times Y\to X$ with $h\circ \iota_{X}\in \AutG(X)\Rightarrow (h,p_{Y})\in \AutG_{Y}(X\times Y)$. Again $$h\circ \iota_{X}\in \AutG(X)\Rightarrow \psi(h\circ i_{X})=(h\circ \iota_{X}\circ p_{X},p_{Y})\in \AutG_{Y}(X\times Y).$$
Since $\phi \big(h,p_{Y}\big)=\phi \big(h\circ \iota_{X}\circ p_{X},p_{Y} \big)$, therefore $(h,p_{Y})\simeq (h\circ \iota_{X}\circ p_{X},p_{Y})$, otherwise $\phi$ fails to be an isomorphism. Consequently we get the desired relation $h\simeq h\circ \iota_{X}\circ p_{X}$.


\end{proof}

We now obtain a description of $\Aut^G(X\times Y)$ in terms of $\Aut^G(X)$ and $\Aut^G(Y)$, which is an equivariant version of Booth and Heath \cite[Theorem 2.7]{Gps}.

\begin{theorem}\label{t2prodexact}
	Let $X$ and $Y$ be connected $G$-CW complexes  and all self-homotopy equivalences of $X\times Y$ are $G$-reducible. If $[Y,\aut_{1}(X)]_{G}=\{ \ast \}$, then we have a split exact sequence $$0\to [X,\aut_{1}(Y)]_{G}\to \AutG(X\times Y)\xrightarrow[]{\phi} \AutG(X)\times \AutG(Y)\to 0.$$ where $\phi(f)=(f_{XX},f_{YY})$.
\end{theorem}

\begin{proof}
First we show that $\phi$ is a homomorphism. Let $f=(f_{X},f_{Y}), g=(g_{X},g_{Y})\in \AutG(X\times Y)$, then 
\begin{align*}
\phi(f\circ g)&=\phi \big(\big[f_{X}\circ (g_{X},g_{Y}),f_{Y}\circ (g_{X},g_{Y})\big]\big)\\
&=\big[f_{X}\circ (g_{X},g_{Y})\circ \iota_{X},f_{Y}\circ (g_{X},g_{Y})\circ \iota_{Y}\big]\\
&=\big[f_{X}\circ \iota_{X}\circ p_{X}\circ (g_{X},g_{Y})\circ \iota_{X},f_{Y}\circ (g_{X}\circ \iota_{X}\circ p_{X},g_{Y})\circ \iota_{Y}\big]~~~(\text{by Lemma }\ref{Lgred})\\
&=(f_{X}\circ \iota_{X}\circ g_{X}\circ \iota_{X},f_{Y}\circ \iota_{Y}\circ g_{Y}\circ \iota_{Y})\\
&=(f_{XX}\circ g_{XX},f_{YY}\circ g_{YY})\\
&=\phi(f)\circ \phi(g).
\end{align*}	

Consider an another map $\psi:\AutG(X)\times \AutG(Y)\to \AutG(X\times Y)$, define as $$\psi(h_{1},h_{2}):=(h_{1}\circ p_{X},h_{2}\circ p_{Y}),~~~~~~\forall (h_{1},h_{2})\in	\AutG(X)\times \AutG(Y).$$
From Proposition \ref{mgred}, we can say that the map $\psi$ is well-defined. Clearly $\phi \circ \psi=\Id_{\AutG(X)\times \AutG(Y)}$, therefore $\phi$ is onto homomorphsim.

Now we determine $\Ker(\phi)$. An element $f=(f_{X},f_{Y})\in \AutG(X\times Y)$ is in $\Ker(\phi)$ if and only if $f_{XX}=\Id_{X},~~~ f_{YY}=\Id_{Y}$. Therefore $f=(f_{X},f_{Y})\in \Ker(\phi)$ corresponds under adjoint associativity to the base point preserving maps $\widetilde{f}_{X}:Y\to \map(X,X)$ and $\widetilde{f}_{Y}:X\to \map(Y,Y)$, defined as $\widetilde{f}_{X}(y)(x):=f_{X}(x,y)$ and $\widetilde{f}_{Y}(x)(y):=f_{Y}(x,y)$ respectively. Since both $X$ and $Y$ are connected, therefore the images of $\widetilde{f}_{X}$ and $\widetilde{f}_{Y}$ are contained in the component of $\map(X,X)$ and $\map(Y,Y)$ which contains  $\Id_{X}$ and $\Id_{Y}$ respectively. These components are denoted by $\aut_{1}(X)$ and $\aut_{1}(Y)$. Again the maps $\widetilde{f}_{X}$ and $\widetilde{f}_{Y}$ are $G$-maps. Therefore the $\Ker(\phi)$ can be identified with $[Y,\aut_{1}(X)]_{G}\times [X,\aut_{1}(Y)]_{G}$. Since $[Y,\aut_{1}(X)]_{G}=\{\ast \}$, we get the required split exact sequence.

\end{proof}

\end{section}
	
\section{$\AutG(X_1\times \ldots  \times X_n)$} \label{snprod}

Now we use the results of the previous section to generalize for higher product of based $G$-CW complexes $X_{1}, X_{2},\ldots,X_{n}$. Consider $$\nprod = X_{1}\times \ldots \times X_{n}$$ with diagonal $G$-action. For a $G$-map $f = (f_1, \ldots, f_n) \colon \nprod \to \nprod$, define $f_{ij}\colon X_j\to X_i$ as the composition  $f_{ij}\colon =p_{i}\circ f\circ \iota_{j}=f_{i}\circ \iota_{j}$, where 
$$p_{i}\colon \nprod \to X_i, ~~~ \iota_{j} \colon X_j \to \nprod, ~~1\leq, i,j \leq n, $$
are projections and inclusions respectively.  Clearly $p_{i}, \iota_{j}$ are $G$-maps, therefore $f_{ij}$ are $G$-maps and the following diagrams are commute.

$$
\xymatrix{
	\nprod \ar[rr]^{f} \ar[rrdd]^{f_i}  &&  \nprod \ar[dd]^{p_{i}} \\\\
	X_{j} \ar@{-->}[rr]_{f_{ij}}\ar[uu]^{\iota_{j}} && X_{i}
}
$$

\begin{definition}
A $G$-self-homotopy equivalence $f$ of $\nprod$ is said to be \emph{$G$-reducible} if the maps $f_{ii} \colon X_i \to X_i$ are $G$-self-homotopy equivalences for all $1\leq i\leq n$.
\end{definition}

Let $\prod_{\hat{i}}:=X_{1}\times \ldots \times \widehat{X_{i}}\times \ldots \times X_{n}$, i.e. $\prod_{\hat{i}}$ denote the subproduct of $\nprod$ obtained by omitting $X_{i}$.
We have natural projection $\nprod \to \prod_{\hat{i}}$. We denote by $\AutG_{\prod_{\hat{i}}}(\nprod)$ the set of $G$-homotopy classes of self-equivalences of $\nprod$ over $\prod_{\hat{i}}.$ 
The following Lemma shows that elements of $\AutG_{\prod_{\hat{i}}}(\nprod)$ are of the form $(p_{1},\ldots,p_{i-1},f_{i},p_{i+1},\ldots, p_{n})$ where $f_{ii} \in \Aut^G(X_i)$.

\begin{lemma} \label{lreplace}
Given a $G$-map $f_i \colon \nprod \to X_i$, we have $f = (p_{1},\ldots,p_{i-1},f_{i},p_{i+1},\ldots, p_{n}) \in \AutG(\nprod)$ iff $f_{ii} \in \Aut^G(X_i)$ for any $1\leq i \leq n.$ Moreover  the inverse of the map $f$ is of the form $(p_{1},\ldots,p_{i-1},\widetilde{f}_{i},p_{i+1},\ldots, p_{n}),$ for some $G$-map $\widetilde{f}_i: \bX \to X_i.$

\end{lemma}

\begin{proof}
Take $\bar{\bf X}= X_{1}\times \ldots \times X_{i-1}\times X_{i+1}\times \ldots \times X_{n}\times X_{i}$.
Consider the permutation map $\phi:\nprod \to \bar{\bf X},$ defined as $(x_{1},\ldots,x_{n}) \mapsto (x_{1},\ldots,x_{i-1},x_{i+1},\ldots,x_{n},x_{i}).$ 

Clearly $\phi$ is a homeomorphism. Therefore $(p_{1},\ldots,p_{i-1},f_{i},p_{i+1},\ldots,p_{n})\in \AutG(\nprod)$ if and only if
$\phi \circ (p_{1},\ldots,p_{i-1},f_{i},p_{i+1},\ldots,p_{n})\circ \phi^{-1}=(p'_{1},\ldots,p'_{i-1},p'_{i+1},\ldots,p'_{n},f'_{i}) \in \AutG(\bar{\bf X}),$
where $p'_{j}:\bar{\bf X}\to X_{j}$ are all projection maps for $j \in \{1,\ldots,i-1,i+1,\ldots,n\}$.

$$
\xymatrix{
	\nprod \ar[dd]_{f}  \ar[rrrr]^{\phi}  &&&&  \bar{\bf X} \ar[dd]^{\phi \circ f\circ \phi^{-1}} \\\\
	\nprod \ar[rrrr]_{\phi} &&&& \bar{\bf X}
}
$$

\noindent Take $\bar{\b X}=X_{1}\times \ldots \times X_{i-1}\times X_{i+1}\times \ldots \times X_{n}$. So $(p'_{\bar{\b X}},f'_{i})=(p'_{1},\ldots,p'_{i-1},p'_{i+1},\ldots,p'_{n},f'_{i})$ and $\bar{\bf X} = \bar{\b X}\times X_{i}$. 
\noindent From Proposition \ref{egred}(b), we can say that the map $(p'_{\bar{\b X}},f'_{i})=(p'_{1},\ldots,p'_{i-1},p'_{i+1},\ldots,p'_{n},f'_{i})\in \AutG(\bar{\bf X})$ if and only if $f'_{ii}\in \AutG(X_{i})$.
Therefore, $$\phi^{-1}\circ (p'_{1},\ldots,p'_{i-1},p'_{i+1},\ldots,p'_{n},f'_{i})\circ \phi = (p_{1},\ldots,p_{i-1},f_{i},p_{i+1},\ldots,p_{n})\in \AutG(\nprod)$$ 
$$\iff f_{ii}\in \AutG(X_{i}),~~~~~~(\because f'_{ii} = f'_{i}\circ \iota'_{i} = f_{i}\circ \phi^{-1}\circ \iota'_{i} = f_{i}\circ \iota_{i} = f_{ii}).$$
Since $(p'_{\bar{\b X}},f'_{i})\in \AutG(\bar{\bf X})$, using Proposition \ref{egred}(a), we can say that $(p'_{\bar{\b X}},\widetilde{f}'_{i})$ is a $G$-homotopy inverse of $(p'_{\bar{\b X}},f'_{i})$ in $\AutG(\bar{\bf X})$. Therefore the map $\phi^{-1}\circ (p'_{\bar{\b X}},\widetilde{f}'_{i})\circ \phi = (p_{1},\ldots,p_{i-1},\widetilde{f}_{i},p_{i+1},\ldots,p_{n})$ is the $G$-inverse of $$\phi^{-1}\circ (p'_{\bar{\b X}},f'_{i})\circ \phi = (p_{1},\ldots,p_{i-1},f_{i},p_{i+1},\ldots,p_{n})$$ in $\AutG(\nprod)$.

\end{proof}
	
We now deduce that $\AutG_{\prod_{\hat{i}}}(\nprod)$ are the subgroups of $\AutG(\nprod)$, for all $i = 1,\ldots,n$.

\begin{corollary}
	For a pointed $G$-CW complex $\nprod$, the subsets $\AutG_{\prod_{\hat{i}}}(\nprod)$ are subgroups of $\Aut^G(\bX)$, for any $i=1,\ldots,n$.
\end{corollary}

\begin{proof}
Any element of $\AutG_{\prod_{\hat{i}}}(\nprod)$ is of the form $(p_1, \ldots, f_i, \ldots, p_n),$ where $f_{i}\circ \iota_{i}\in \AutG(X_{i})$. This set is non-empty because $(p_{1},\ldots,p_{n})\in \AutG_{\prod_{\hat{i}}}(\nprod)$. Let for any two elements $(p_1, \ldots, f_i, \ldots, p_n), ~~~(p_1, \ldots, g_i, \ldots, p_n)\in \AutG_{\prod_{\hat{i}}}(\nprod)$, we have
$$(p_1, \ldots, f_i, \ldots, p_n)\circ (p_1, \ldots, g_i, \ldots, p_n) = (p_1, \ldots, f_i\circ (p_1, \ldots, g_i, \ldots, p_n),\ldots, p_n).$$ Clearly, $$f_{i}\circ (p_1, \ldots, g_i, \ldots, p_n)\circ \iota_{i} = f_{ii}\circ g_{ii}\in \AutG(X_{i}).$$ Therefore $(p_1, \ldots, f_i, \ldots, p_n)\circ (p_1, \ldots, g_i, \ldots, p_n)\in \AutG_{\prod_{\hat{i}}}(\nprod)$.
From Lemma \ref{lreplace}, we have, $(p_1, \ldots, f_i, \ldots, p_n)^{-1} = (p_1, \ldots, \widetilde{f}_i, \ldots, p_n)$, which is in $\AutG_{\prod_{\hat{i}}}(\nprod)$. Consequently $\AutG_{\prod_{\hat{i}}}(\nprod)$ becomes a subgroup of $\AutG(\nprod)$, for any $i=1,\ldots,n$.
	
\end{proof}

The following theorem and its corollary give a sufficient condition for an equivariant map $f\colon \nprod\to \nprod$ to be a $G$-self-homotopy equivalence under the $G$-reducibility assumption.
	
\begin{theorem}\label{tgred}
Assume that all $G$-self-homotopy equivalences of $\nprod$ are $G$-reducible (where each $X_{i}$ are G-CW complexes). Given maps $f_{i}\in \map^{G}(\nprod,  X_{i})$ for $k\leq i\leq n$, we have $(p_{1},\ldots, p_{k-1}, f_{k},\ldots, f_{n})\in \AutG(\nprod)$  if and only if $f_{ii}\in \AutG(X_{i})$ for all $k\leq i \leq n$.
\end{theorem}

\begin{proof}
Let $(p_{1},\ldots,p_{k-1},f_{k},\ldots,f_{n})\in \AutG(\nprod)$, therefore by using $G$-reducibility we have $f_{ii}\in \AutG(X_{i})$.

Conversely, $k =n $ follows from the above Lemma \ref{lreplace}. Now we consider the case $k = n-1.$ Observe that, 
$$(p_{1},\ldots,p_{n-2},f_{n-1},f_{n})=\Big[p_{1},\ldots,p_{n-1},f_{n}\circ (p_{1},\ldots,f_{n-1},p_{n})^{-1}\Big] \circ (p_{1},\ldots,f_{n-1},p_{n}).$$
\noindent By Lemma \ref{lreplace}, the inverse $(p_{1},\ldots,f_{n-1},p_{n})^{-1}$ is of the form $ (p_{1},\ldots,\widetilde{f}_{n-1},p_{n}) \in \Aut^G(\nprod).$ Since $f_{nn}\in \Aut^G(X_n)$, we get by using Lemma \ref{lreplace} again, 
\begin{align*}
(p_{1},\ldots,p_{n-1},f_{n})\circ (p_{1},\ldots,\widetilde{f}_{n-1},p_{n}) =\Big[p_{1},\ldots,\widetilde{f}_{n-1},f_{n}\circ &(p_{1},....,\widetilde{f}_{n-1},p_{n})\Big]
\in \AutG(\nprod).
\end{align*}
Using $G$-reducibility we have $f_{n}\circ (p_{1},\ldots,\widetilde{f}_{n-1},p_{n})\circ \iota_{n}\in \AutG(X_{n})$. Therefore $\Big[p_{1},\ldots,p_{n-1},f_{n}\circ (p_{1},\ldots,\widetilde{f}_{n-1},p_{n})\Big]\in \AutG(\nprod)$.
Consequently,
$$(p_{1},\ldots,p_{n-2},f_{n-1},f_{n})=\Big[p_{1},\ldots,p_{n-1},f_{n}\circ (p_{1},\ldots,\widetilde{f}_{n-1},p_{n})\Big]\circ (p_{1},\ldots,f_{n-1},p_{n}),$$
is in $ \AutG(\nprod)$ whenever $f_{n-1n-1}\in \AutG(X_{n-1})$ and $f_{nn}\in \AutG(X_{n})$. This completes the proof for $k = n-1.$

\noindent Now consider the case $k = n-2$. We have factorisation, 
\begin{align*}
& (p_{1},\ldots,p_{n-3},f_{n-2},f_{n-1},f_{n}) = \\
& \Big[p_{1},\ldots,p_{n-2},f_{n-1}\circ (p_{1},\ldots,\widetilde{f}_{n-2},p_{n-1},p_{n}),
f_{n}\circ (p_{1},\ldots,\widetilde{f}_{n-2},p_{n-1},p_{n}) \Big] \circ \\
& (p_{1},\ldots,f_{n-2},p_{n-1},p_{n}).
\end{align*}
Clearly, the second factor is in $\Aut^G(\nprod)$ by Lemma \ref{lreplace}. To show the first factor is in $\Aut^G(\nprod)$, it is enough to show that (by the previous step $k = n-1$),
$$f_{n-1}\circ (p_{1},\ldots,\widetilde{f}_{n-2},p_{n-1},p_{n})\circ \iota_{n-1}\in \AutG(X_{n-1}),$$ and $$f_{n}\circ (p_{1},\ldots,\widetilde{f}_{n-2},p_{n-1},p_{n})\circ \iota_{n}\in \AutG(X_{n}).$$
Since,
\begin{align*}
& (p_{1},\ldots,p_{n-2},f_{n-1},f_{n})\circ (p_{1},\ldots,p_{n-3},\widetilde{f}_{n-2},p_{n-1},p_{n}) = \\
&\Big[p_{1},\ldots,\widetilde{f}_{n-2},f_{n-1}\circ (p_1,\ldots,\widetilde{f}_{n-2},p_{n-1},p_{n}), f_{n}\circ (p_{1},\ldots,\widetilde{f}_{n-2},p_{n-1},p_{n})\Big] \in \AutG(\nprod).
\end{align*}
Using $G$-reducibility we conclude that,
$$f_{n-1}\circ (p_{1},\ldots,p_{n-3},\widetilde{f}_{n-2},p_{n-1},p_{n})\circ \iota_{n-1}\in \AutG(X_{n-1}),$$ 
and 
$$f_{n}\circ (p_{1},\ldots,p_{n-3},\widetilde{f}_{n-2},p_{n-1},p_{n})\circ \iota_{n}\in \AutG(X_{n}).$$
%
%
%
Therefore, $(p_{1},\ldots,p_{n-3},f_{n-2},f_{n-1},f_{n})\in \AutG(\nprod)$, whenever $f_{ii}\in \AutG(X_{i}),~~~~\forall i=n-2,n-1,n$.

\noindent Inductively, we can prove that $(p_{1},\ldots,p_{k-1},f_{k},\ldots,f_{n})\in \AutG(\nprod)$, whenever $f_{ii}\in \AutG(X_{i}), ~~k\leq i \leq n$.

\end{proof}

\begin{corollary}\label{Sher}
Assume that all $G$-self-homotopy equivalences of $\nprod$ are $G$-reducible. Then an equivariant map $f = (f_{1},\ldots,f_{n})\colon \nprod \to \nprod$ is $G$-self homotopy equivalence if $f_{ii}\in \AutG(X_{i})$ for all $1\leq i\leq n$.
\end{corollary}

\begin{proof}
Let $f = (f_{1},\ldots,f_{n})\colon \nprod\to \nprod$ be an equivariant map such that $f_{ii}\in \AutG(X_{i})$ for all $1\leq i\leq n.$ Therefore from Theorem \ref{tgred} we have $(p_{1},f_{2},\ldots,f_{n})\in \AutG(\nprod)$.
\noindent Now,
$$(f_{1},\ldots,f_{n})=\Big[p_{1},f_{2}\circ (\widetilde{f}_{1},p_{2},\ldots,p_{n}),\ldots,f_{n}\circ (\widetilde{f}_{1},p_{2},\ldots,p_{n})\Big]
\circ (f_{1},p_{2},\ldots,p_{n}).$$

\noindent By Lemma  \ref{lreplace}, we have $(f_{1},p_{2},\ldots,p_{n})\in \AutG(\nprod)$ with its inverse $(\widetilde{f}_{1},p_{2},\ldots,p_{n})$.
Since,
$$(p_{1},f_{2},\ldots,f_{n})\circ (\widetilde{f}_{1},p_{2},\ldots,p_{n})\\
=\Big[\widetilde{f}_{1},f_{2}\circ (\widetilde{f}_{1},p_{2},\ldots,p_{n}),\ldots,f_{n}\circ (\widetilde{f}_{1},p_{2},\ldots,p_{n})\Big]$$
in $\AutG(\nprod)$, using $G$-reducibility we can say that
$$f_{j}\circ (\widetilde{f}_{1},p_{2},\ldots,p_{n})\circ \iota_{j}\in \AutG(X_{j}), ~~~\forall ~~~ 2\leq j \leq n.$$
Therefore from Theorem \ref{tgred} we have, $$\Big[p_{1},f_{2}\circ (\widetilde{f}_{1},p_{2},\ldots,p_{n}),\ldots,f_{n}\circ (\widetilde{f}_{1},p_{2},\ldots,p_{n})\Big]\in \AutG(\nprod).$$
Consequently,  $$(f_{1},\ldots,f_{n})=\Big[p_{1},f_{2}\circ (\widetilde{f}_{1},p_{2},\ldots,p_{n}),\ldots,f_{n}\circ (\widetilde{f}_{1},p_{2},\ldots,p_{n})\Big]\circ (f_{1},p_{2},\ldots,p_{n})$$ in $\AutG(\nprod)$.

\end{proof}

\begin{proposition}\label{Rtgred}
	
Assume that all $G$-self homotopy equivalences of $\nprod$ are $G$-reducible. 
%
Let $W=X_1\times\cdots\times X_k$ and $Z=X_{k+1}\times \cdots \times X_n.$ If $f = (f_W, f_Z) \in \Aut^G(\bX)$ then $f_{WW} \in\Aut^G(W)$, $f_{ZZ} \in\Aut^G(Z)$. 
In particular the product spaces $W, Z$ are also $G$-reducible. 
\end{proposition}

\begin{proof}
Let $f\in \Aut^{G}(\nprod)$.  Therefore $f=(f_W,f_Z)\in \Aut^{G}(W\times Z).$ From Proposition \ref{pgred}(a), it is sufficient to show that $f_{ZZ}\in \Aut^{G}(Z)$.

Since $f\in \Aut^{G}(\nprod)$, Theorem \ref{tgred} gives $(p_W,f_Z)\in \Aut^{G}(W\times Z).$  Consider the fiber sequence,

$$
\xymatrix{
Z^{H} \ar[dd]_{f^{H}_{ZZ}}  \ar[rr]^{(\iota_Z)^{H}}  &&  (W\times Z)^{H} \ar[dd]^{(p_W,f_Z)^{H}} \ar[rr]^{(p_W)^{H}} && W^{H}\ar[dd]^{Id_W} \\\\
Z^{H} \ar[rr]_{(\iota_Z)^{H}} && (W\times Z)^{H} \ar[rr]^{(p_W)^{H}} && W^{H}
}
$$
Using Five-Lemma corresponding to the long exact sequence of homotopy groups of the above fiber sequence, we can say that $f^{H}_{ZZ}\in \Aut(Z^{H})$, for all $H\leq G$ i.e. $f_{ZZ}\in \Aut^{G}(Z).$


Hence all $G$-self homotopy equivalences of $W\times Z$ are $G$-reducible.	
\end{proof}

We now show that the group $\AutG(\nprod)$ is the product of its subgroups $\Aut_{\prod_{\hat{i}}}^{G}(\nprod)$, where $i = 1,\ldots,n$.

\begin{theorem}\label{tnprod}
Assume that all $G$-self-homotopy equivalences of $\nprod$ are $G$-reducible. Then,  $$\AutG(\nprod)=\Aut_{\prod_{\hat{n}}}^{G}(\nprod)\ldots \Aut_{\prod_{\hat{1}}}^{G}(\nprod).$$
Moreover, if each $X_i$ are connected then $\Aut_{\prod_{\hat{i}}}^{G}(\nprod)$ fits into a split exact sequence 
$$0\to [\Pi_{\hat{i}},\aut_{1}(X_i)]_{G} \to \Aut_{\prod_{\hat{i}}}^{G}(\nprod)\xrightarrow[]{\phi_{i}} \AutG(X_{i}) \to 0,$$ defined by $$\phi_{i}\big(p_{1},\ldots,f_{i},\ldots,p_{n}\big) =  f_{i}\circ \iota_{i},~~~\forall i=1,\ldots,n$$
\end{theorem}

\begin{proof}
We first claim that for any $2\leq k \leq n-1$  $$\Aut_{X_{1}\times \ldots \times X_{k-1}}^{G}(\nprod)=\Aut_{X_{1}\times \ldots \times X_{k}}^{G}(\nprod)\cdot \Aut_{\prod_{\hat{k}}}^{G}(\nprod).$$
Since $\Aut_{X_{1}\times \ldots \times X_{k}}^{G}(\nprod)\cap \Aut_{\prod_{\hat{k}}}^{G}(\nprod)=(p_{1},\ldots,p_{n})$, therefore if the factorization exists, it must be unique.
Now, 
\begin{align*}
&(p_{1},\ldots,p_{k-1},f_{k},\ldots,f_{n})=\\
&\Big[p_{1},\ldots,p_{k},f_{k+1}\circ (p_{1},\ldots,p_{k-1},\widetilde{f}_{k},p_{k+1},\ldots,p_{n}),\ldots,
f_{n}\circ (p_{1},\ldots,p_{k-1},\widetilde{f}_{k},p_{k+1},\ldots,p_{n})\Big]\circ \\ &(p_{1},\ldots,p_{k-1},f_{k},p_{k+1},\ldots,p_{n}).
\end{align*}
Second factor $(p_{1},\ldots,p_{k-1},f_{k},p_{k+1},\ldots,p_{n})\in \Aut_{\prod_{\hat{k}}}^{G}(\nprod)$ by Lemma \ref{lreplace}.
 Again,
\begin{align*}
&(p_{1},\ldots,p_{k},f_{k+1},\ldots,f_{n})\circ (p_{1},\ldots,p_{k-1},\widetilde{f}_{k},p_{k+1},\ldots,p_{n})=\\
&\Big[p_{1},\ldots,\widetilde{f}_{k},f_{k+1}\circ (p_{1},\ldots,p_{k-1},\widetilde{f}_{k},p_{k+1},\ldots,p_{n}),
\ldots,f_{n}\circ (p_{1},\ldots,p_{k-1},\widetilde{f}_{k},p_{k+1},\ldots,p_{n})\Big]\\
\end{align*}
in $\AutG(\nprod).$ Using $G$-reducibility we have,
$$f_{j}\circ (p_{1},\ldots,p_{k-1},\widetilde{f}_{k},p_{k+1},\ldots,p_{n})\circ \iota_{j}\in \AutG(X_{j}),~~~\forall ~~ k+1 \leq j\leq n.$$
Therefore the first factor,
$$\Big[p_{1},\ldots,p_{k},f_{k+1}\circ (p_{1},\ldots,p_{k-1},\widetilde{f}_{k},p_{k+1},\ldots,p_{n}),
\ldots,f_{n}\circ (p_{1},\ldots,p_{k-1},\widetilde{f}_{k},p_{k+1},\ldots,p_{n})\Big]$$ in $\Aut_{X_{1}\times \ldots \times X_{k}}^{G}(\nprod)$,
from Theorem \ref{tgred}. Hence we have proved our claim.

\noindent Similarly we can prove that, $$\AutG(\nprod) = \Aut_{X_{1}}^{G}(\nprod)\cdot \Aut_{\prod_{\hat{1}}}^{G}(\nprod),$$ by using the factorization $$(f_{1},\ldots,f_{n})=\Big[p_{1},f_{2}\circ (\widetilde{f}_{1},p_{2},\ldots,p_{n}),\ldots,f_{n}\circ (\widetilde{f}_{1},p_{2},\ldots,p_{n})\Big]\circ (f_{1},p_{2},\ldots,p_{n}).$$
Hence,
\begin{align*}
& \AutG(\nprod)=\Aut_{X_{1}}^{G}(\nprod)\cdot \Aut_{\prod_{\hat{1}}}^{G}(\nprod)\\
&=\Aut_{X_{1}\times X_{2}}(\nprod)\cdot \Aut_{\prod_{\hat{2}}}^{G}(\nprod)\cdot \Aut_{\prod_{\hat{1}}}^{G}(\nprod)\\
&= \ldots \\
&=\Aut_{\prod_{\hat{n}}}^{G}(\nprod)\ldots \Aut_{\prod_{\hat{1}}}^{G}(\nprod).
\end{align*}
For the second part of the theorem, we consider the group homomorphism $$\phi_{i}\colon \Aut_{\prod_{\hat{i}}}^{G}(\nprod)\to \AutG(X_{i}),~~ \big(p_{1},\ldots,f_{i},\ldots,p_{n}\big) \mapsto f_{i}\circ \iota_{i}.$$
It has a splitting map $$\psi_{i}\colon \AutG(X_{i})\to \Aut_{\prod_{\hat{i}}}^{G}(\nprod) ~~ g  \mapsto \big(p_{1},\ldots,g\circ p_{i},\ldots,p_{n}).$$ From Lemma \ref{lreplace}, we can say that the map $\psi_{i}$ is well-defined. Therefore we get a split short exact sequence $$0\to \Ker(\phi_{i})\to \Aut_{\prod_{\hat{i}}}^{G}(\nprod)\xrightarrow[]{\phi_{i}} \AutG(X_{i}) \to 0.$$
Since each $X_{i}$ are connected, therefore $\prod_{\hat{i}}$ is connected and similar to \ref{egred} (d) we can say that $\Ker(\phi)$ is identified with $\big[\Pi_{\hat{i}},\aut_{1}(X_i)\big]_{G}$.
\end{proof}

\section{LU-factorization}\label{slu}

In the previous section we have described a factorization of $\Aut^G(X_1 \times \ldots \times X_n)$ which is symmetric. However, the number of factors in the decomposition increases with the number of factors in the topological product, which is inconvenient. In this section we give a different factorisation with always only two factors, called $LU$ factorisation. It is closely related $LU$ decomposition of matrices. We begin by introducing some notations. Define maps $u_{k}:\bX \to \bX$ and $l_{k}: \bX \to \bX$ on the product space $\nprod = X_1 \times \ldots \times X_n$:
$$l_{k}(x_{1},\ldots,x_{n}):=(x_{1},\ldots,x_{k},*, \ldots, *), ~~u_{k}(x_{1},\ldots,x_{n}):=(*,\ldots.*,x_{k},\ldots, x_{n}).$$
Note that $$l_{j}\circ l_{k}=l_{k}\circ l_{j}=l_{\min\{j,k\}}, ~~u_{j}\circ u_{k}=u_{k}\circ u_{j}=u_{\max\{j,k\}}.$$
Let us define two subsets $L^{G}(X_{1},\ldots,X_{n})$ and $U^{G}(X_{1},\ldots,X_{n})$ by

$$L^{G}(X_{1},\ldots,X_{n}):=\{f\in \AutG(\nprod): f_{k}=f_{k}\circ l_{k}, ~~1\leq k \leq n\},$$ 

$$U^{G}(X_{1},\ldots,X_{n}):=\{f\in \AutG(\nprod): f_{k}=f_{k}\circ u_{k},~~f_{kk}=\Id_{X_{k}},  ~~1\leq k\leq n\}.$$
The induced endomorphism on homotopy $f_{\#i}^H\colon  \pi_{i}(X_{1}^H)\times \ldots \times  \pi_{i}(X_{n}^H)\to  \pi_{i}(X_{1}^H)\times \ldots \times \pi_{i}(X_{n}^H)$ can be represented as matrix form for all $H\leq G$,

$$M_{i}(f^{H}):=
\begin{bmatrix}
(f^{H}_{11})_{\#i}& (f^{H}_{12})_{\#i}& \ldots  (f^{H}_{1n})_{\#i} \\
(f^{H}_{21})_{\#i}& (f^{H}_{22})_{\#i}& \ldots  (f^{H}_{2n})_{\#i} \\
.                                                                                                                                 \\
.                                                                                                                                  \\
.                                                                                                                                   \\
(f^{H}_{n1})_{\#i} & (f^{H}_{n2})_{\#i}& \ldots (f^{H}_{nn})_{\#i}
\end{bmatrix}.$$

If $f\in L^{G}(X_{1},\ldots,X_{n})$, then 
$$f_{kj}=f_{k}\circ \iota_{j}= f_{k}\circ l_{k}\circ \iota_{j}= 0 ~~~~~\forall ~~ k<j ~~~ (\because l_{k}\circ \iota_{j}=0,\forall k<j).$$


Therefore the matrix can be written as: 
		
$$M_{i}(f^{H})=
\begin{bmatrix}
(f^{H}_{11})_{\#i}& 0&0& \ldots 0 \\
(f^{H}_{21})_{\#i}& (f^{H}_{22})_{\#i}&0& \ldots 0 \\
.                                                                                                                                 \\
.                                                                                                                                  \\
.                                                                                                                                   \\
(f^{H}_{n1})_{\#i} & (f^{H}_{n2})_{\#i}&(f^{H}_{n3})_{\#i}& \ldots (f^{H}_{nn})_{\#i}
\end{bmatrix}.$$

\noindent Hence the induced map of each element in $L^{G}(X_{1},\ldots,X_{n})$ gives the lower triangular matrices on homotopy groups of every fixed point spaces.
Similarly, for $g\in U^{G}(X_{1},\ldots,X_{n})$ we have,

$$ g_{kj} = g_{k}\circ \iota_{j} = g_{k}\circ u_{k}\circ \iota_{j} = 0 ~~~~~~\forall ~~ k>j , ~~~~~(\because u_{k}\circ \iota_{j}=0,\forall~~ k>j).  $$

\noindent Therefore the matrix can be written as:
		
$$M_{i}(g^{H}):=
\begin{bmatrix}
1&(g^{H}_{12})_{\#i}& (g^{H}_{13})_{\#i}&\ldots (g^{H}_{1n})_{\#i} \\
0&1& (g^{H}_{23})_{\#i}&\ldots (g^{H}_{2n})_{\#i} \\
.                                                                                                                                 \\
.                                                                                                                                  \\
.                                                                                                                                   \\
0 & 0&0&\ldots 1
\end{bmatrix}.$$

\noindent Hence induced map of each element in $U^{G}(X_{1},\ldots,X_{n})$ gives upper triangular matrices with identity automorphisms on the diagonal entries on homotopy groups of every fixed point spaces.

As in the non-equivariant case, we have the following chain of inclusions:
$$L^{G}(X_{1},\ldots,X_{n})\subseteq L^{G}(X_{1}\times X_{2},X_{3},\ldots,X_{n})\subseteq \ldots \subseteq L^{G}(\nprod)=\AutG(\nprod),$$
		
$$\AutG(\nprod) \supseteq U^{G}(X_{1},\ldots,X_{n})\supseteq U^{G}(X_{1}\times X_{2},X_{3},\ldots,X_{n})\supseteq \ldots \supseteq U^{G}(\nprod)=\{ 1 \} .$$

\noindent Now we show that these two subsets are actually the subgroups of $\AutG(\nprod)$.

\begin{proposition}\label{Sgred}
Assume that all $G$-self-homotopy equivalences of  $\nprod$ are $G$-reducible. Then $L^{G}(X_{1},\ldots,X_{n})$  and  $U^{G}(X_{1},\ldots,X_{n})$ are subgroups of $\AutG(\nprod)$.
\end{proposition}

\begin{proof}

(a) We first consider the case of $L^{G}(X_{1},\ldots,X_{n})$. Let $f,g\in L^{G}(X_{1},\ldots,X_{n})$. Then $p_{j}\circ f=f_{j}=f_{j}\circ l_{j};  ~ p_{j}\circ g=g_{j}=g_{j}\circ l_{j}$.

\noindent For any $j\leq k$, we know that $l_{j}\circ l_{k}=l_{k}\circ l_{j}=l_{j}$.
Now,
\begin{align*}
l_{k}\circ g\circ l_{k}&=l_{k}\circ (g_{1},\ldots,g_{k},\ldots,g_{n})\circ l_{k}
=l_{k}\circ  (g_{1}\circ l_{k},\ldots,g_{k}\circ l_{k},\ldots,g_{n}\circ l_{k})\\
&=(g_{1}\circ l_{k},\ldots,g_{k}\circ l_{k},*,\ldots,*)
=(g_{1}\circ l_{1}\circ l_{k},\ldots,g_{k}\circ l_{k}\circ l_{k},*,\ldots,*)\\
&=(g_{1}\circ l_{1},\ldots,g_{k}\circ l_{k},*,\ldots,*)
=(g_{1},\ldots,g_{k},*,\ldots,*)
=l_{k}\circ g.\\
&(\because  l_{j}\circ l_{k}=l_{\min\{j,k\}})
\end{align*}

\noindent So,
$$p_{k}\circ (f\circ g)=p_{k}\circ f\circ l_{k}\circ g=p_{k}\circ f\circ l_{k}\circ g\circ l_{k}=p_{k}\circ f\circ g\circ l_{k}
=p_{k}\circ (f\circ g)\circ l_{k}.$$

\noindent Hence $f,g\in L^{G}(X_{1},\ldots,X_{n})$ implies $f\circ g\in L^{G}(X_{1},\ldots,X_{n})$.
To prove $L^{G}(X_{1},\ldots,X_{n})$ is closed under formation of inverses, we proceed by induction. Let us first consider $L^{G}(X_{1},X_{2})$. We can decompose an element $f\in L^{G}(X_{1},X_{2})$ as  $$f=(f_{11}\circ p_{1},p_{2})\circ [(f_{11}\circ p_{1},p_{2})^{-1}\circ f].$$
We split the proof into several claims. 
\vspace{.25 cm}

\noindent {\bf Claim I:} $(f_{11}\circ p_{1},p_{2})^{-1}=(f^{-1}_{11}\circ p_{1},p_{2})$. 
\vspace{.25 cm}

\noindent Observe that, 

\begin{align*}
(f_{11}\circ p_{1},p_{2})\circ (f^{-1}_{11}\circ p_{1},p_{2}) &
=[f_{11}\circ p_{1}\circ (f^{-1}_{11}\circ p_{1},p_{2}),p_{2}]
=[f_{11}\circ p_{1}(f^{-1}_{11}\circ p_{1},p_{2}),p_{2}] \\
& =(f_{11}\circ f^{-1}_{11}\circ p_{1},p_{2})
=(p_{1},p_{2})
=\Id_{X_{1}\times X_{2}}.
\end{align*}

\noindent Similarly, $(f^{-1}_{11}\circ p_{1},p_{2})\circ (f_{11}\circ p_{1},p_{2})=\Id_{X_{1}\times X_{2}}$. Hence the claim is true. 

\noindent Since $f\in L^{G}(X_{1},X_{2})$ we have $(f_{11}\circ p_{1},p_{2})\in \AutG(X_{1}\times X_{2})$, by using $G$-reducibility and Proposition \ref{egred}. Therefore $$(f^{-1}_{11}\circ p_{1},p_{2})=(f_{11}\circ p_{1},p_{2})^{-1}\in \AutG(X_{1}\times X_{2}).$$
\noindent Again, $f^{-1}_{11}\circ p_{1}=f^{-1}_{11}\circ p_{1}\circ l_{1},~~~~ p_{2}=p_{2}\circ l_{2}$.
\noindent Therefore $(f^{-1}_{11}\circ p_{1},p_{2})\in L^{G}(X_{1},X_{2}),$ which imply again $(f^{-1}_{11}\circ p_{1},p_{2})\circ f\in L^{G}(X_{1},X_{2})$.
\noindent Since,

\begin{align*}
p_{1}\circ (f^{-1}_{11}\circ p_{1},p_{2})\circ f &=p_{1}\circ (f^{-1}_{11}\circ p_{1},p_{2})\circ (f_{1},f_{2})
=p_{1}\circ [f^{-1}_{11}\circ p_{1}(f_{1},f_{2}),f_{2}]\\
&=p_{1}\circ (f^{-1}_{11}\circ f_{1},f_{2})
=f^{-1}_{11}\circ f_{1}
=f^{-1}_{11}\circ f_{11}
=p_{1}.\\
&~~~~~~(\because f_{1}=f_{1}\circ l_{1}=f_{1}\circ \iota_{1}).
\end{align*}
Therefore $(f^{-1}_{11}\circ p_{1},p_{2})\circ f \in \AutG_{X_{1}}(X_{1}\times X_{2})\subseteq L^{G}(X_{1}\times X_{2})$.
Since $\AutG_{X_{1}}(X_{1}\times X_{2})$ is a group and  $(f^{-1}_{11}\circ p_{1},p_{2})\circ f \in \AutG_{X_{1}}(X_{1}\times X_{2}),$ we have,
\begin{align*}
&\Rightarrow f^{-1}\circ (f^{-1}_{11}\circ p_{1},p_{2})^{-1}\in \AutG_{X_{1}}(X_{1}\times X_{2})\subseteq L^{G}(X_{1},X_{2})\\
&\Rightarrow  f^{-1}\circ (f^{-1}_{11}\circ p_{1},p_{2})^{-1}\circ  (f^{-1}_{11}\circ p_{1},p_{2})\in L^{G}(X_{1},X_{2})\\
&\Rightarrow f^{-1}\in L^{G}(X_{1}, X_{2}).
\end{align*}
Hence $L^{G}(X_{1},X_{2})$ becomes a group.

Now let $f\in L^{G}(X_{1},\ldots,X_{n})\subseteq L^{G}(X_{1}\times X_{2},\ldots,X_{n})$, we can assume inductively that $f^{-1}\in  L^{G}(X_{1}\times X_{2},\ldots,X_{n})$.
\noindent Then we get the conditions $$p_{j}\circ f^{-1}=p_{j}\circ f^{-1}\circ l_{j},~~~ 2 \leq j \leq n.$$ So it remains to prove that $p_{1}\circ f^{-1}=p_{1}\circ f^{-1}\circ l_{1}$, which is equivalent to show that $p_{12}\circ f^{-1}\circ \iota_{12}\in L^{G}(X_{1},X_{2})$. This is because,

\begin{align*}
&p'_{1}\circ (p_{12}\circ f^{-1}\circ \iota_{12})=p'_{1}\circ (p_{12}\circ f^{-1}\circ \iota_{12})\circ l'_{1}\\
\Leftrightarrow & p'_{1}\circ p_{12}\circ f^{-1}\circ \iota_{12}=p_{1}\circ f^{-1}\circ l_{1}\circ \iota_{12}\\
\Leftrightarrow & p'_{1}\circ p_{12}\circ f^{-1}\circ \iota_{12}\circ p_{12}=p_{1}\circ f^{-1}\circ l_{1}\circ \iota_{12}\circ p_{12}\\
\Leftrightarrow & p'_{1}\circ p_{12}\circ f^{-1}\circ l_{12}=p_{1}\circ f^{-1}\circ l_{1}\circ l_{12}\\
\Leftrightarrow & p'_{1}\circ p_{12}\circ f^{-1}=p_{1}\circ f^{-1}\circ l_{1},~~~~~(\because f^{-1}\in L^{G}(X_{1}\times X_{2},\ldots,X_{n}))\\
\Leftrightarrow & p_{1}\circ f^{-1}=p_{1}\circ f^{-1}\circ l_{1}.
\end{align*}
See the commutative diagrams:
			
$$
\xymatrix{
X_{1}\times \ldots \times X_{n} \ar[rrdd]_{p_{1}}  \ar[rr]^{p_{12}}  &&   X_{1}\times X_{2} \ar[dd]^{p'_{1}} \\\\
&& X_{1}
}
$$

$$
\xymatrix{
X_{1}\times X_{2} \ar[dd]_{\iota_{12}}  \ar[rr]^{l'_{1}}  &&   X_{1}\times X_{2} \ar[dd]^{\iota_{12}} \\\\
X_{1}\times \ldots \times X_{n} \ar[rr]_{l_{1}} && X_{1}\times \ldots \times X_{n}
}
$$
			
and,

$$
\xymatrix{
X_{1}\times \ldots \times X_{n} \ar[rrdd]_{l_{12}}  \ar[rr]^{p_{12}}  &&   X_{1}\times X_{2} \ar[dd]^{\iota_{12}} \\\\
&& X_{1}\times \ldots \times X_{n}
}
$$
		
\noindent From Proposition \ref{Rtgred} we have $p_{12}\circ f\circ \iota_{12}\in \AutG(X_{1}\times X_{2})$.  Now,

\begin{align*}
 p'_{1}\circ (p_{12}\circ f\circ \iota_{12})=p_{1}\circ f\circ \iota_{12}
=p_{1}\circ f\circ l_{1}\circ \iota_{12} 
 =p_{1}\circ f\circ \iota_{12}\circ l'_{1}\\ 
=p'_{1}\circ (p_{12}\circ f\circ \iota_{12})\circ l'_{1},
\end{align*}
and,
 $$p'_{2}\circ (p_{12}\circ f\circ \iota_{12})=p'_{2}\circ (p_{12}\circ f\circ \iota_{12})\circ l'_{2},~~~(\because l'_{2} = \Id_{X_{1}\times X_{2}}).$$
 
\noindent Therefore $p_{12}\circ f\circ \iota_{12}\in L^{G}(X_{1},X_{2})$.

\vspace{.25 cm}
\noindent{\bf Claim II:} $(p_{12}\circ f\circ \iota_{12})^{-1}=p_{12}\circ f^{-1}\circ \iota_{12}$.
\vspace{.25 cm}

\noindent Note that,
\begin{align*}
(p_{12}\circ f^{-1}\circ \iota_{12})\circ (p_{12}\circ f\circ \iota_{12})
& = p_{12}\circ f^{-1}\circ \iota_{12}\circ p_{12}\circ f\circ \iota_{12}
= p_{12}\circ f^{-1}\circ l_{12}\circ f\circ \iota_{12}\\
& = p_{12}\circ f^{-1}\circ f\circ \iota_{12}
= p_{12}\circ \iota_{12}
= \Id_{X_{1}\times X_{2}}.\\
&~~~~~(\because f^{-1}\in L^{G}(X_{1}\times X_{2},\ldots,X_{n}))
\end{align*}
Similarly, $(p_{12}\circ f\circ \iota_{12})\circ (p_{12}\circ f^{-1}\circ \iota_{12})=\Id_{X_{1}\times X_{2}}$.
Therefore $ (p_{12}\circ f^{-1}\circ \iota_{12})=(p_{12}\circ f\circ \iota_{12})^{-1}\in L^{G}(X_{1}\times X_{2})$. So that $f^{-1}\in L^{G}(X_{1},\ldots,X_{n})$.
\noindent Consequently, $L^{G}(X_{1},\ldots,X_{n})$ is closed under taking inverses. Hence  $L^{G}(X_{1},\ldots,X_{n})$ becomes a group.

\vspace{0.5 cm}
(b)
Now we want to prove that $U^{G}(X_{1},\ldots,X_{n})$ is a group.
Let us first prove that $U^{G}(X_{1},X_{2})$ is a group and then we proceed by induction.
\noindent If $f\in U^{G}(X_{1},X_{2})$  then we have $f_{2}=f_{2}\circ u_{2}$ and $f_{11}=\Id_{X_{1}},~~~f_{22}=\Id_{X_{2}}$.
Now,
\begin{align*}
f_{2} =f_{2}\circ u_{2}
 =f_{2}\circ \iota_{2}\circ p_{2}
 =f_{22}\circ p_{2}
 =p_{2}
~~~~~~(\because u_{2}=\iota_{2}\circ p_{2}).
\end{align*}
This implies that $f\in \AutG_{X_{2}}(X_{1}\times X_{2})$ (cf. Proposition \ref{egred}). Thus we have shown that $$U^{G}(X_{1},X_{2})\subseteq \AutG_{X_{2}}(X_{1},X_{2}).$$ From Proposition \ref{egred}, we have a group homomorphsim:

$$\phi:\AutG_{X_{2}}(X_{1}\times X_{2})\to \AutG(X_{1}) ,~~ (f_{1},p_{2}) \mapsto f_{1}\circ \iota_{1}.$$
Observe that, $U^{G}(X_{1},X_{2})=\Ker(\phi)$. Hence $U^{G}(X_{1},X_{2})$ is a group. 
We now proceed by induction on $n$. Assume that $U^{G}(X_{1},\ldots,X_{n-1})$ is a group.

\vspace{.25 cm}
\noindent {\bf Claim III:}  $f\in U^{G}(X_{1},\ldots,X_{n})\Rightarrow p_{1 \ldots (n-1)}\circ f\circ \iota_{1 \ldots (n-1)}\in U^{G}(X_{1},\ldots,X_{n-1})$.
\vspace{.25 cm}

\noindent  From Propostion \ref{Rtgred} we have $p_{1 \ldots (n-1)}\circ f\circ \iota_{1 \ldots (n-1)}\in \AutG(X_{1}\times \ldots \times X_{n-1})$. Now,
\begin{align*}
 p'_{j}\circ p_{1 \ldots (n-1)}\circ f\circ \iota_{1 \ldots (n-1)}\circ \iota'_{j}&=p_{j}\circ f\circ \iota_{j}
 =f_{jj}
=\Id_{X_{j}}, ~~~~~~~\forall j=1,\ldots,n-1.
\end{align*}
And,
\begin{align*}
 p'_{j}\circ p_{1 \ldots (n-1)}\circ f\circ \iota_{1 \ldots (n-1)}\circ u'_{j} =p_{j}\circ f\circ u_{j}\circ \iota_{1 \ldots (n-1)}
=p_{j}\circ f\circ \iota_{1 \ldots (n-1)}\\
=p'_{j}\circ p_{1 \ldots (n-1)}\circ f\circ \iota_{1 \ldots (n-1)}
~~~~\forall j=1,\ldots,n-1.
\end{align*}
			
Here the maps $p_j', \iota_j', u_j' $ are defined by the following commutative diagrams:

$$
\xymatrix{
	X_{1}\times \ldots \times X_{n} \ar[rrdd]_{p_{j}}  \ar[rr]^{p_{1\ldots (n-1)}}  &&   X_{1}\times \ldots \times X_{n-1} \ar[dd]^{p'_{j}} \\\\
	&& X_{j}
}
$$

$$
\xymatrix{
	X_{1}\times \ldots \times X_{n-1} \ar[rr]^{\iota_{1\ldots (n-1)}}  &&  \nprod  \\\\
	X_{j} \ar[rruu]_{\iota_{j}}\ar[uu]^{\iota'_{j}} 
}
$$ 

$$
\xymatrix{
X_{1}\times \ldots \times X_{n-1} \ar[dd]_{\iota_{1 \ldots (n-1)}}  \ar[rr]^{u'_{j}}  &&   X_{1}\times \ldots \times X_{n-1} \ar[dd]^{\iota_{1 \ldots (n-1)}} \\\\
\nprod \ar[rr]_{u_{j}} && \nprod
}
$$
			
\noindent Therefore the claim is proved.

From Proposition \ref{egred}, we have the split short exact sequence,

$$0\to \Ker(\phi)\xrightarrow[]{\iota} \AutG_{X_{n}}(\nprod)\xrightarrow[]{\phi} \AutG(X_{1}\times \ldots \times X_{n-1})\to 0,$$

\noindent where $$\phi(f):= p_{1 \ldots (n-1)}\circ f\circ \iota_{1 \ldots (n-1)} = f_{1 \ldots (n-1)}\circ \iota_{1 \ldots (n-1)},$$ and $$\Ker(\phi) = \big\{f\in \AutG_{X_{n}}(\nprod): p_{1 \ldots (n-1)}\circ f\circ \iota_{1 \ldots (n-1)} =\Id_{X_{1}\times \ldots \times X_{n-1}}\big\}.$$
Now, $$\psi:\AutG(X_{1}\times \ldots \times X_{n-1})\to \AutG_{X_{n}}(\nprod)$$ is the splitting map of $\phi$ defined as $\psi(h):=(h\circ p_{1 \ldots (n-1)},p_{n})$.
\noindent Observe that, $$f\in U^{G}(X_{1},\ldots,X_{n}) \Rightarrow f_{n}=f_{n}\circ u_{n} \text{  and  } ~~ f_{nn}=\Id_{X_{n}}.$$
Therefore, $f_{n}=f_{n}\circ u_{n}=f_{n}\circ \iota_{n}\circ p_{n}=f_{nn}\circ p_{n}=p_{n}$. Hence $f\in \AutG_{X_{n}}(\nprod)$, i.e. $U^{G}(X_{1},\ldots, X_{n})\subseteq \AutG_{X_{n}}(\nprod)$.
\vspace{.25 cm}

\noindent {\bf Claim IV:} $\Ker(\phi) = U^{G}(X_{1}\times \ldots \times X_{n-1},X_{n}).$
 \vspace{.25 cm}

\noindent Note that,
\begin{align*}
&f\in U^{G}(X_{1}\times \ldots \times X_{n-1},X_{n})\\
&\Rightarrow p_{1 \ldots (n-1)}\circ f\circ \iota_{1 \ldots (n-1)} = \Id_{X_{1}\times \ldots \times X_{n-1}}; ~f_{n}=f_{n}\circ u_{n} , ~f_{nn}=\Id_{X_{n}}\\
&\Rightarrow p_{1 \ldots (n-1)}\circ f\circ \iota_{1 \ldots (n-1)} = \Id_{X_{1}\times \ldots \times X_{n-1}}; f_{n}=p_{n}\\
&\Rightarrow f\in \Ker(\phi).
\end{align*}
Conversely,
\begin{align*}
&g\in \Ker(\phi)\\
&\Rightarrow p_{1\ldots (n-1)}\circ g\circ \iota_{1 \ldots (n-1)}=\Id_{X_{1}\times...\times X_{n-1}};~~ g_{n} = p_{n}\\
&\Rightarrow p_{1\ldots (n-1)}\circ g\circ \iota_{1 \ldots (n-1)}=\Id_{X_{1}\times...\times X_{n-1}};~~ g_{n} = g_{n}\circ u_{n},~ g_{nn} = \Id_{X_{n}}\\
&\Rightarrow g\in U^{G}(X_{1}\times \ldots \times X_{n-1},X_{n}).
\end{align*}
So  $\Ker(\phi)=U^{G}(X_{1}\times \ldots \times X_{n-1},X_{n})\subseteq U^{G}(X_{1},\ldots,X_{n})$. Therefore we get the split short exact sequence,
$$0\to U^{G}(X_{1}\times \ldots \times X_{n-1},X_{n})\xrightarrow []{\iota} U^{G}(X_{1},\ldots,X_{n})\xrightarrow [] {\phi} U^{G}(X_{1},\ldots,X_{n-1})\to 0.$$
Hence $U^{G}(X_{1},\ldots,X_{n})$ is the semi-direct product of $U^{G}(X_{1}\times \ldots \times X_{n-1},X_{n})$ with $U^{G}(X_{1},\ldots,X_{n-1}).$ Hence it is a group.
\end{proof}

Finally we have the symmetric decomposition of $\AutG(\nprod)$ in the following which is an equivariant version of \cite[Theorem 5.4]{PRshe} 
\begin{theorem}\label{tlu}
Assume that all the $G$-self-homotopy equivalences of $\nprod$ are $G$-reducible. Then 
$$\AutG(\nprod)=L^{G}(X_{1},\ldots,X_{n})\cdot U^{G}(X_{1},\ldots,X_{n}).$$
\end{theorem}

\begin{proof}
Let $f\in L^{G}(X_{1},\ldots,X_{n})\cap U^{G}(X_{1},\ldots,X_{n}).$ Then $$f_{j}=f_{j}\circ l_{j};~~ f_{j}=f_{j}\circ u_{j},~~~ f_{jj}=\Id_{X_{j}},~~ \forall ~~1 \leq j \leq n .$$
Therefore,
\begin{align*}
&~f_{j}=(f_{j}\circ l_{j})\circ u_{j}~~~~(\because f_{j}=f_{j}\circ l_{j})\\
\Rightarrow &~f_{j}=f_{j}\circ \iota_{j}\circ p_{j}~~~~~(\because l_{j}\circ u_{j}=\iota_{j}\circ p_{j})\\
\Rightarrow &~f_{j}=f_{jj}\circ p_{j}\\
\Rightarrow &~f_{j}=p_{j}\\
\Rightarrow &~f=(p_{1},\ldots,p_{n})=\Id_{\nprod}.
\end{align*}
Therefore  $L^{G}(X_{1},\ldots,X_{n})\cap U^{G}(X_{1},\ldots,X_{n})$ is trivial.

We begin the proof with $n=2$, then argue by induction. A map $f\in \AutG(X_{1}\times X_{2})$ can be decomposed as, $$f=[f\circ (f_{1},p_{2})^{-1}\circ (f_{11}\circ p_{1},p_{2})]\circ [(f_{11}\circ p_{1},p_{2})^{-1}\circ (f_{1},p_{2})].$$ From Proposition \ref{Sgred}, we have $(f_{11}\circ p_{1},p_{2})\in L^{G}(X_{1},X_{2})$.
\noindent For $f\in \AutG(X_{1}\times X_{2})$ we have $(f_{1},p_{2})\in \AutG(X_{1}\times X_{2})$, from Corollary \ref{lgred}.
Again $$f\circ (f_{1},p_{2})^{-1}=[f_{1}\circ(f_{1},p_{2})^{-1},f_{2}\circ (f_{1},p_{2})^{-1}].$$
Note that $(f_{1},p_{2})\circ  (f_{1},p_{2})^{-1}=(p_{1},p_{2})\Rightarrow f_{1}\circ  (f_{1},p_{2})^{-1}=p_{1}$.
Therefore $$f\circ (f_{1},p_{2})^{-1}=[p_{1},f_{2}\circ (f_{1},p_{2})^{-1}]\in L^{G}(X_{1},X_{2}).$$
This implies the first factor in the representation of $f$ is in $L^{G}(X_{1},X_{2})$.
\noindent On the other hand, $$(f_{11}\circ p_{1},p_{2})^{-1}\circ (f_{1},p_{2})\in U^{G}(X_{1},X_{2}),$$ because
\begin{align*}
p_{2}\circ (f_{11}\circ p_{1},p_{2})^{-1}\circ (f_{1},p_{2})\circ u_{2}
=&~p_{2}\circ (f^{-1}_{11}\circ p_{1},p_{2})\circ (f_{1},p_{2})\circ u_{2}\\
=&~p_{2}\circ (f^{-1}_{11}\circ p_{1},p_{2})\circ (f_{1}\circ u_{2},p_{2}\circ u_{2})\\
=&~p_{2}\circ (f^{-1}_{11}\circ p_{1},p_{2})\circ (f_{1}\circ u_{2},p_{2})\\
=&~p_{2}\circ [f^{-1}_{11}\circ p_{1}\circ (f_{1}\circ u_{2},p_{2}),p_{2}]\\
=&~p_{2}\\
=&~p_{2}\circ (f^{-1}_{11}\circ p_{1},p_{2})\circ (f_{1},p_{2}).
\end{align*}
Again,
\begin{align*}
p_{1}\circ (f_{11}\circ p_{1},p_{2})^{-1}\circ (f_{1},p_{2})\circ \iota_{1}
=&~p_{1}\circ (f^{-1}_{11}\circ p_{1},p_{2})\circ (f_{1}\circ \iota_{1},p_{2}\circ \iota_{1})\\
=&~p_{1}\circ (f^{-1}_{11}\circ p_{1},p_{2})\circ (f_{11},p_{2}\circ \iota_{1})\\
=&~p_{1}\circ [f^{-1}_{11}\circ p_{1}\circ (f_{11},p_{2}\circ \iota_{1}),p_{2}\circ \iota_{1}]\\
=&~p_{1}\circ (f^{-1}_{11}\circ f_{11},p_{2}\circ \iota_{1})\\
=&~f^{-1}_{11}\circ f_{11}\\
=&~\Id_{X_{1}}.
\end{align*}
And,
\begin{align*}
p_{2}\circ (f_{11}\circ p_{1},p_{2})^{-1}\circ (f_{1},p_{2})\circ \iota_{2}
=&~p_{2}\circ (f^{-1}_{11}\circ p_{1},p_{2})\circ (f_{1}\circ \iota_{2},p_{2}\circ \iota_{2})\\
=&~p_{2}\circ (f^{-1}_{11}\circ p_{1},p_{2})\circ (f_{1}\circ \iota_{2},Id_{X_{2}})\\
=&~p_{2}\circ [f^{-1}_{11}\circ p_{1}\circ (f_{1}\circ \iota_{2},Id_{X_{2}}),Id_{X_{2}}]\\
=&~\Id_{X_{2}}.
\end{align*}
Therefore $(f_{11}\circ p_{1},p_{2})^{-1}\circ (f_{1},p_{2})\in U^{G}(X_{1},X_{2})$.

Let $f\in \AutG(X_{1}\times \ldots \times X_{n})$. If we take $W=X_{1}\times X_{2}$ then $f\in \AutG(W\times X_{3}\times \ldots \times X_{n}).$ By induction hypothesis, there exists $f'\in L^{G}(W,X_{3},\ldots,X_{n})$ and $f''\in U^{G}(W,X_{3},\ldots,X_{n})$ such that $f=f'\circ f''$.
\noindent Since $f'\in L^{G}(W,  X_{3},\ldots,X_{n})\subseteq \AutG(\nprod)$ we have, $p_{12}\circ f'\circ \iota_{12}\in \AutG(W)$ by Proposition \ref{Rtgred}. Let $\bar{f}=p_{12}\circ f'\circ \iota_{12}$, then $\bar{f}=l\circ \hat{u}$ for some $l\in L^{G}(X_{1},X_{2})$ and $\hat{u}\in U^{G}(X_{1},X_{2})$.
Let $$\bar{u}:=(\hat{u}\circ p_{12},p_{3},\ldots,p_{n})=(\hat{u}_{1}\circ p_{12},\hat{u}_{2}\circ p_{12},p_{3},\ldots,p_{n}).$$

\noindent{\bf Claim:} $\bar{u}\in U^{G}(X_{1},\ldots,X_{n}).$

\begin{align*}
\hat{u}_{1}\circ p_{12}\circ \iota_{1}&=\hat{u}_{1}\circ \iota'_{1} 
= \hat{u}_{11}
= \Id_{X_{1}}
~~~(\because p_{12}\circ \iota_{1}=\iota'_{1}).
\end{align*}
\begin{align*}
\hat{u}_{2}\circ p_{12}\circ \iota_{2} =\hat{u}_{2}\circ \iota'_{2}
= \hat{u}_{22}
= \Id_{X_{2}}
~~~~(\because p_{12}\circ \iota_{2}=\iota'_{2}).
\end{align*}
$$p_{{j}}\circ i_{{j}}= \Id_{X_{j}},  ~~3\leq j\leq n.$$
And
\begin{align*}
\hat{u}_{2}\circ p_{12}\circ u_{2}&=\hat{u}_{2}\circ u'_{2}\circ p_{12}~~~~~~(\because p_{12}\circ u_{2}=u'_{2}\circ p_{12})\\
&=\hat{u}_{2}\circ p_{12}~~~(\because \hat{u}=(\hat{u}_{1},\hat{u}_{2})\in U^{G}(X_{1},X_{2})).
\end{align*}
$$p_{j}\circ u_{j}= p_{j},  ~~3\leq j \leq n.$$
Therefore the claim is proved. 
\noindent Clearly $\bar{u}^{-1}=(u^{-1}\circ p_{12},p_{3},\ldots,p_{n})\in L^{G}(X_{1}\times X_{2},\ldots,X_{n})$. Hence $f'\circ \bar{u}^{-1}\in L^{G}(X_{1}\times X_{2},\ldots,X_{n})$.
Since $f'\in L^{G}(X_{1}\times X_{2},\ldots,X_{n})$, we have,
$$p_{12}\circ f'=p_{12}\circ f'\circ l_{12}=p_{12}\circ f'\circ \iota_{12}\circ p_{12}~~~~~~(\because \iota_{12}\circ p_{12}=l_{12}).$$
To prove $f'\circ \bar{u}^{-1}\in L^{G}(X_{1},\ldots,X_{n})$, it is sufficient to show that $p_{12}\circ (f'\circ \bar{u}^{-1})\circ \iota_{12}\in L^{G}(X_{1}\times X_{2})$. Now,

$$
p_{12}\circ f'\circ \bar{u}^{-1}\circ \iota_{12} =p_{12}\circ f'\circ \iota_{12}\circ p_{12}\circ \bar{u}^{-1}\circ \iota_{12} =p_{12}\circ f' \circ \iota_{12}\circ u^{-1} =\bar{f}\circ u^{-1} =l.
$$

\noindent Therefore $p_{12}\circ f'\circ \bar{u}^{-1}\circ \iota_{12}\in L^{G}(X_{1}\times X_{2})$. Consequently we have $f'\circ \bar{u}^{-1}\in L^{G}(X_{1},\ldots,X_{n})$. Clearly, $\bar{u}\circ f''\in U^{G}(X_{1},\ldots,X_{n})$. Hence $f=(f'\circ \bar{u}^{-1})\circ (\bar{u}\circ f'')$.
			
\end{proof}

Let $$\bar{L}^{G}(X_{1},\ldots,X_{n}):=\{f\in L^{G}(X_{1},\ldots,X_{n})| ~ f_{kk}=\Id_{X_{k}}, ~~1\leq k \leq n \}.$$

\begin{proposition}
The group $L^{G}(X_{1},\ldots,X_{n})$ is isomorphic to the semi-direct product of $\AutG(X_{1})\times \ldots \times \AutG(X_{n})$ with  $\bar{L}^{G}(X_{1},\ldots,X_{n})$.
\end{proposition}

\begin{proof}
Consider a map $$\varphi:L^{G}(X_{1},\ldots,X_{n})\to \AutG(X_{1})\times \ldots \times \AutG(X_{n}),$$ defined as $$f\mapsto (f_{11},\ldots,f_{nn}).$$

Now,
\begin{align*}
(f\circ g)_{jj}&=p_{j}\circ (f\circ g)\circ \iota_{j}\\
&=p_{j}\circ f\circ (g_{1},\ldots,g_{n})\circ \iota_{j}\\
&=p_{j}\circ f\circ (g_{1}\circ l_{1},\ldots,g_{n}\circ l_{n})\circ \iota_{j}\\
&=p_{j}\circ f\circ (g_{1}\circ l_{1}\circ i_{j},\ldots,g_{n}\circ l_{n}\circ \iota_{j})\\
&=p_{j}\circ f\circ (*,\ldots,*,g_{j}\circ \iota_{j},\ldots,g_{n}\circ \iota_{j})~~~(\because l_{k}\circ \iota_{j}=*, \forall k<j)\\
&=f_{j}\circ (*,\ldots,*,g_{j}\circ \iota_{j},\ldots,g_{n}\circ \iota_{j})\\
&=f_{j}\circ l_{j}\circ (*,\ldots,*,g_{j}\circ \iota_{j},\ldots,g_{n}\circ \iota_{j})\\
&=f_{j}\circ (*,\ldots,*,g_{j}\circ \iota_{j},*,\ldots,*)\\
&=f_{j}\circ \iota_{j}\circ g_{jj}  =f_{jj}\circ g_{jj}.
\end{align*}
		
Therefore,
\begin{align*}
\varphi (f\circ g) =((f\circ g)_{11},\ldots,(f\circ g)_{nn}) 
 =(f_{11}\circ g_{11},\ldots,f_{nn}\circ g_{nn})  \\
 =(f_{11},\ldots,f_{nn})\circ (g_{11},\ldots,g_{nn}) 
 =\varphi (f)\circ \varphi (g).
\end{align*}

So,
\begin{align*}
\Ker(\varphi)&=\big\{ f\in L^{G}(X_{1},\ldots,X_{n})|~~f_{jj} = \Id_{X_{j}},~\forall ~~1\leq j \leq n \big\}\\
&=\bar{L}^{G}(X_{1},\ldots,X_{n}).
\end{align*}
Now, $$\Psi:\AutG(X_{1})\times \ldots \times \AutG(X_{n})\to L^{G}(X_{1},\ldots,X_{n})$$ is the splitting map of $\varphi$ defined as $$\Psi (g_{1},\ldots,g_{n}):=(g_{1}\circ p_{1},\ldots,g_{n}\circ p_{n}).$$ Using the Corollary \ref{Sher}, we can say that the map $\Psi$ is well-defined.  Therefore we get a split short exact sequence, $$0\to \bar{L}^{G}(X_{1},\ldots,X_{n})\to L^{G}(X_{1},\ldots,X_{n})\xrightarrow[]{\varphi} \AutG(X_{1})\times \ldots \times \AutG(X_{n})\to 0.$$ Hence $L^{G}(X_{1},\ldots,X_{n})\cong (\AutG(X_{1})\times \ldots \times \AutG(X_{n}))\rtimes \bar{L}^{G}(X_{1},\ldots,X_{n})$

\end{proof}

\end{document}